\documentclass[10pt,letterpaper,reqno]{amsart}
\usepackage{amssymb}
\usepackage{amsmath}
\usepackage{amsrefs}
\usepackage{amsthm}
\usepackage{array}
\usepackage{rotating}
\usepackage{verbatim}
\usepackage{enumitem} 

\theoremstyle{plain}
\newtheorem{theorem}[equation]{Theorem}
\newtheorem{corollary}[equation]{Corollary}

\newtheorem{proposition}[equation]{Proposition}

\theoremstyle{definition}

\theoremstyle{remark}

\newtheoremstyle{indenteddefinition}{\topsep}{\topsep}{\addtolength{\leftskip}{2.0em}}{-0em}{\bfseries}{.}{
}{}
\theoremstyle{indenteddefinition}
\newtheorem{definition}[equation]{Definition}

\DeclareMathOperator\Lie{Lie}

\DeclareMathOperator\Ad{Ad}

\DeclareMathOperator\Aut{Aut}

\DeclareMathOperator\Hom{Hom}

\DeclareMathOperator\Max{Max}

\DeclareMathOperator\Ind{Ind}

\DeclareMathOperator\End{End}

\DeclareMathOperator\tr{tr}

\DeclareMathOperator\RE{Re}

\DeclareMathOperator\diag{diag}
\DeclareMathOperator\Spec{Spec}
\DeclareMathOperator\Spin{Spin}
\DeclareMathOperator\spin{spin}
\DeclareMathOperator\sgn{sgn}

\newcommand\Tstrut{\rule{0pt}{2.6ex}}
\newcommand\Bstrut{\rule[-0.9ex]{0pt}{0pt}} 

\begin{document}
\author{Henrik Schlichtkrull}
\address[Henrik Schlichtkrull]{Department of Mathematics, University of
Copenhagen\\ Universitetsparken 5, DK-2100 Copenhagen \O}
\email{schlicht@math.ku.dk}
\author{Peter Trapa}\thanks{The second author was supported in part by
    NSF grant DMS-1302237.}
\address[Peter Trapa]{Department of Mathematics, University of Utah\\
  Salt Lake City, UT 84112}
\email{ptrapa@math.utah.edu}
\author{David A. Vogan, Jr.}
\address[David Vogan]{Department of Mathematics, MIT \\ Cambridge, MA 02139}
\email{dav@math.mit.edu}
\title{Laplacians on spheres}
\date{\today}
\maketitle
\begin{abstract}
Spheres can be written as homogeneous spaces $G/H$ for compact Lie
groups in a small number of ways. In each case, the decomposition of
$L^2(G/H)$ into irreducible representations of $G$ contains
interesting information. We recall these decompositions, and see what
they can reveal about the analogous problem for noncompact real forms
of $G$ and $H$.
\end{abstract}

\bigskip
\noindent This paper is dedicated to Joe Wolf, in honor of all that we
have learned from him about the connections among geometry,
representation theory, and harmonic analysis; and in gratitude for
wonderful years of friendship.

\smallskip
\section{Introduction}\label{sec:intro}
\setcounter{equation}{0}
The sphere has a Riemannian metric, unique up to a positive scale,
that is preserved by the action of the orthogonal group. Computing the
spectrum of the Laplace operator is a standard and beautiful
application of representation theory. These notes will look at some
variants of this computation, related to interesting subgroups of the
orthogonal group.

The four variants presented in Sections \ref{sec:R}, \ref{sec:C},
\ref{sec:H}, and \ref{sec:O} correspond to the following very general
fact, due to \'Elie Cartan: if $G/K$ is an irreducible compact
Riemannian symmetric space of real rank $1$, then $K$ is
transitive on the unit sphere in $T_{eK}(G/K)$. (The only caveat is in
the case of the one-dimensional symmetric space ${S^1}$. In this case
one needs to use the full isometry group $O(2)$ rather than its
identity component to get the transitivity.) The isotropy group of
a point on the sphere is often called $M$ in the theory; so the
conclusion is that
\begin{equation}\label{e:rank1}
  \text{sphere of dimension ($\dim G/K -1$)} \simeq K/M
\end{equation}
The calculations we do correspond to the rank one symmetric spaces
\begin{alignat*}{5}
  &O(n+1)/O(n), && S^{n-1} \simeq O(n)/O(n-1) \qquad &&\text{(Section
    \ref{sec:R}),}\\
  &SU(n+1)/U(n),&& S^{2n-1} \simeq U(n)/U(n-1) &&\text{(Section \ref{sec:C}),}\\[.5ex]
  &Sp(n+1)/[Sp(n)\times Sp(1)],\quad &&\begin{matrix} S^{4n-1} \simeq
    [Sp(n)\times Sp(1)]/ \\[.2ex] \qquad [Sp(n-1)\times
      Sp(1)_\Delta] \end{matrix} &&\text{(Section
    \ref{sec:H}), and}\\[.5ex]
  &F_4/\Spin(9) && S^{15} \simeq \Spin(9)/\Spin(7)' &&\text{(Section
    \ref{sec:O}).}
\end{alignat*}

The representations of $O(n)$, $U(n)$, $Sp(n)\times Sp(1)$, and
$\Spin(9)$ that we are computing are exactly the $K$-types of the
spherical principal series representations for the noncompact forms of
the symmetric spaces.

Rank one symmetric spaces provide three infinite families (and one
exceptional example) of realizations of spheres as homogeneous spaces
(for compact Lie groups). A theorem due to Montgomery-Samelson and
Borel (\cite{MS} and \cite{BOct}; there is a nice account in
\cite{Wolf}*{(11.3.17)}) classifies {\em all} such realizations. In
addition to some minor variants on those above, like
$$S^{2n-1} \simeq SU(n)/SU(n-1),\qquad S^{4n-1} \simeq
Sp(n)/Sp(n-1),$$
the only remaining possibilities are
$$\begin{aligned}
   S^6 \simeq G_{2,c}/SU(3) \qquad &\text{(Section \ref{sec:G2}), and}\\
  S^7 \simeq \Spin(7)/G_{2,c} \qquad &\text{(Section \ref{sec:bigG2}).}\\
\end{aligned}$$

After recalling in Sections \ref{sec:R}--\ref{sec:bigG2} the classical
harmonic analysis related to these various realizations of spheres, we
will examine in Sections \ref{sec:invt}--\ref{sec:size} what these
classical results say about invariant differential operators.

In Sections \ref{sec:Opq}--\ref{sec:bigncG2} we examine what
this information about harmonic analysis on spheres can tell us about
harmonic analysis on hyperboloids. With $n=p+q$ the symmetric spaces
$$H_{p,q}=O(p,q)/O(p-1,q),\quad (0\le q\le n)$$
are said to be {\it real forms} of each other (and thus in
particular of
$$S^{n-1}=O(n)/O(n-1) = H_{n,0}).$$
Similarly, each of the realizations listed above of $S^{n-1}$ as a
non-symmetric homogeneous space for a subgroup of $O(n)$ corresponds
to one or more noncompact real forms, realizing some of the $H_{p,q}$
as non-symmetric homogeneous spaces for subgroups of $O(p,q)$. These
realizations exhibit the hyperbolic spaces as examples of real
spherical spaces of rank one, and as such our interest is primarily
with their discrete series.  These and related spaces have previously been
studied by T.~Kobayashi (see \cite{Kobayashi-Stiefel,Kob,toshi:zuckerman,toshi:howe}).  In Sections
\ref{sec:Upq}--\ref{sec:bigncG2} we give an essentially self-contained treatment,
in some cases giving slightly more refined information. In particular, we obtain some
interesting discrete series representations for small parameter values
for the real forms of 
$S^6 \simeq G_{2,c}/SU(3)$.

\begin{subequations}\label{se:orbitmethod}
For information about real spherical spaces and their discrete series
in general we refer to \cite{KKOS}; this paper was intended in part to
examine some interesting examples of those results. In particular, we
are interested in formulating the parametrization of discrete series
in a way that may generalize as much as possible. We are very grateful
to Job Kuit for extensive discussions of this parametrization problem.

One such formulation involves the ``method of coadjoint orbits:''
representations of $G$ are parametrized by certain orbits $G\cdot
\lambda$ of $G$ on the real dual vector space
\begin{equation}\label{eq:coadjoint}
  {\mathfrak g}_0^* =_{\text{def}} \Hom_{\mathbb R}(\Lie(G),{\mathbb R})
\end{equation}
(often together with additional data). The orbits corresponding to
representations appearing in $G/H$
typically have representatives
\begin{equation}\label{eq:GmodH}
  \lambda \in [{\mathfrak g}_0/{\mathfrak h}_0]^*.
\end{equation}
We mention this at the beginning of the paper because this coadjoint
orbit parametrization is often {\em not} a familiar one (like that of
representations of compact groups by highest weights). We will write
something like
\begin{equation}\label{eq:orbitparam}
  \pi(\text{orbit\ }\lambda,\Lambda)
\end{equation}
for the representation of $G$ parametrized by $G\cdot\lambda$ (and
sometimes additional data $\Lambda$). If $G$ is an equal-rank reductive group
and $\lambda \in {\mathfrak g}_0^*$ is a regular elliptic element
(never mind exactly what these terms mean), then
\begin{equation}\label{eq:dsparam}
  \pi(\text{orbit\ }\lambda) = \text{discrete series with
    Harish-Chandra parameter $i\lambda$;}
\end{equation}
so this looks like a moderately familiar parametrization. (Here
``discrete series representation'' has the classical meaning of an
irreducible summand of $L^2(G)$. Soon we will use the term more
generally to refer to summands of $L^2(G/H)$.) But notice
that \eqref{eq:dsparam} includes the case of $G$ compact. In that case
$\lambda$ is not the highest weight, but rather an exponent in the
Weyl character formula.

Here is how most of our discrete series will arise. Still for $G$
reductive, if $\lambda$ is elliptic but possibly singular, define
\begin{equation}\label{eq:cohindA}
  G^\lambda = L, \qquad {\mathfrak q} = {\mathfrak l} + {\mathfrak u}
\end{equation}
to be the $\theta$-stable parabolic subalgebra defined by the
requirement that
\begin{equation}\label{eq:cohindB}
  i\lambda(\alpha^\vee) > 0, \qquad (\alpha \in \Delta({\mathfrak
    u},{\mathfrak h})).
\end{equation}
The ``additional data'' that we sometimes need is a one-dimensional character
\begin{equation}\label{eq:cohindC}
  \Lambda \colon L \rightarrow {\mathbb C}^\times, \quad d\Lambda =
  i\lambda + \rho({\mathfrak u}).
\end{equation}
(If $G^\lambda$ is connected, which is automatic if $G$ is connected
and $\lambda$ is elliptic, then $\Lambda$ is uniquely determined by
$\lambda$; the {\em existence} of $\Lambda$ is an {\em integrality}
constraint on $\lambda$.) Attached to $(\lambda,\Lambda)$ is a
cohomologically induced unitary 
representation $\pi(\text{orbit\ }\lambda,\Lambda)$ satisfying
  \begin{equation}\label{eq:cohindD}\begin{aligned}
      \text{infinitesimal character} &= i\lambda - \rho_L = d\Lambda
      -\rho.\\[1ex]
      \text{lowest $K$-type} &= \Lambda - 2\rho({\mathfrak u}\cap
           {\mathfrak k}) \\
           &= i\lambda - \rho({\mathfrak u}\cap
           {\mathfrak p}) + \rho({\mathfrak u}\cap {\mathfrak k}).
    \end{aligned}
  \end{equation}
  If $\lambda$ is small, the formula for the lowest $K$-type can fail:
  one thing that is true is that this representation of $K$ appears if
  the weight is dominant for $K$.

In \cite{VZ}, the representation $\pi(\text{orbit\ }\lambda,\Lambda)$
was called $A_{\mathfrak q}(\Lambda-2\rho({\mathfrak u}))$.

If $G = K$ is compact, then
  \begin{equation}\label{eq:cohindcpt}
   \pi(\text{orbit\ }\lambda) = \text{repn of highest
     weight\ } i\lambda + \rho({\mathfrak 
     u}).
  \end{equation}
If this weight fails to be dominant, then (still in the
compact case)  $\pi(\text{orbit\ }\lambda,\Lambda) = 0$. A confusing
but important aspect of this construction is that the same
representation of $G$ may be attached to several different coadjoint
orbits. Still for $G=K$ compact, the trivial representation is attached
to the orbit of $i\rho({\mathfrak u})$ for each of the
($2^{\text{semisimple rank($K$)}}$) different $K$ conjugacy classes of parabolic
subalgebras ${\mathfrak q}$. If we are looking at the trivial
representation inside functions on a homogeneous space $G/H$, then the
requirement \eqref{eq:GmodH} will ``prefer'' only some of these orbits:
different orbits for different $H$.
\end{subequations} 

\begin{subequations}\label{se:invts}
{\bf Notational convention.} If $(\pi,V_\pi)$ is a representation of a
group $G$, and $H\subset G$ is a subgroup, we write
\begin{equation}\label{eq:invts}
  (\pi^H,V_\pi^H),
\end{equation}
or often just $\pi^H$ for the subspace of $H$-fixed vectors in
$V_\pi$. If $T\in \End(V_\pi)$ preserves $V_\pi^H$, then we will write
\begin{equation}\label{eq:invtops}
  \pi^H(T) =_{\text{def}} T|_{V_\pi^H}
\end{equation}
for the restriction of $T$ to the invariant vectors. This notation
may be confusing because we often write a family of representations
of $G$ as something like
\begin{equation}
  \{\pi^G_s \mid s\in S\};
\end{equation}
then in the notation $[\pi^G_s]^H$, the superscripts $G$ and $H$ have
entirely different meanings. We hope that no essential ambiguity
arises in this way.
\end{subequations} 
\section{The classical calculation}
\label{sec:R}
\setcounter{equation}{0}

\begin{subequations}\label{se:Rsphere}
Suppose $n\ge 1$ is an integer. Write $O(n)$ for the orthogonal group
of the standard inner product on ${\mathbb R}^n$, and
\begin{equation}
  S^{n-1} = \{v\in {\mathbb R}^n \mid \langle v,v\rangle = 1\}
\end{equation}
for the $(n-1)$-dimensional sphere. We choose as a base point
\begin{equation}
  e_1 = (1,0,\ldots,0) \in S^{n-1},
\end{equation}
which makes sense by our assumption that $n\ge 1$.
Then $O(n)$ acts transitively on $S^{n-1}$, and the isotropy group at
$e_1$ is
\begin{equation}
  O(n)^{e_1} \simeq O(n-1);
\end{equation}
we embed $O(n-1)$ in $O(n)$ by acting on the last $n-1$
coordinates. This shows
\begin{equation}
  S^{n-1} \simeq O(n)/O(n-1).
\end{equation}
\end{subequations} 

Now Frobenius reciprocity guarantees that if $H\subset G$ are compact
groups, then
\begin{equation}
  L^2(G/H) \simeq \sum_{(\pi,V_\pi)\in \widehat G} V_\pi \otimes
  (V_\pi^*)^H.
\end{equation}
In words, the multiplicity of an irreducible representation $\pi$ of
$G$ in $L^2(G/H)$ is equal to the dimension of the space of $H$-fixed
vectors in $\pi^*$. So understanding functions on $G/H$ amounts to
understanding representations of $G$ admitting an $H$-fixed
vector. All of the compact homogeneous spaces $G/H$ that we will
consider are {\em Gelfand pairs}, meaning that $\dim (V_{\pi^*})^H \le
1$ for every $\pi \in \widehat G$.

\begin{subequations}\label{se:Rreps}
Here's how that looks for our example. We omit the cases $n=1$ and $n=2$,
which are degenerate versions of the same thing; so assume $n\ge 3$. A
maximal torus in
$O(n)$ is
\begin{equation}
  T = SO(2)^{[n/2]},
\end{equation}
so a weight is an $[n/2]$-tuple of integers. For every integer $a
\ge 0$ there is an irreducible representation $\pi^{O(n)}_a$ of highest
weight
\begin{equation}\label{eq:Rdim}
  (a,0,\ldots,0), \qquad \dim \pi^{O(n)}_a = \frac{(a+n/2
    -1)\prod_{j=1}^{n-3} (a+j)}{(n/2 -1)\cdot (n-3)!}.
\end{equation}
Notice that the polynomial function of $a$ giving the dimension has
degree $n-2$. One natural description of $\pi^{O(n)}_a$ is
\begin{equation}\label{eq:harmonic}
  \pi^{O(n)}_a = S^a({\mathbb C}^n)/r^2 S^{a-2}({\mathbb C}^n);
\end{equation}
what we divide by is zero if $a< 2$. We will be interested in the
{\em infinitesimal characters} of the representations $\pi^{O(n)}_a$; that
is, the scalars by which elements of
\begin{equation}
  {\mathfrak Z}({\mathfrak o}(n)) =_{\text{def}} U({\mathfrak
    o}(n)_{\mathbb C})^{O(n)}
\end{equation}
act on $\pi^{O(n)}_a$. According to Harish-Chandra's theorem,
infinitesimal characters may be identified with Weyl group orbits of
complexified weights.  The infinitesimal character of a
finite-dimensional representation of highest weight $\lambda$ is given
by $\lambda+\rho$, with $\rho$ half the sum of the positive
roots. Using the calculation of $\rho$ given in \eqref{eq:Orho}, we get
\begin{equation}\label{eq:Oinfchar}
  \text{infinitesimal character}(\pi^{O(n)}_a) = (a+(n-2)/2, (n-4)/2,
  (n-6)/2,\cdots).
\end{equation}

The key fact (in the notation explained in \eqref{se:invts}) is that
\begin{equation}\label{eq:Rkey}
  \dim [\pi^{O(n)}_a]^{O(n-1)} = 1 \quad (a\ge 0), \qquad \dim\pi^{O(n-1)} = 0
  \quad (\pi \not\simeq \pi^{O(n)}_a).
\end{equation}
Therefore
\begin{equation}\label{eq:Osphere}
  L^2(S^{n-1}) \simeq \sum_{a=0}^\infty \pi^{O(n)}_a
\end{equation}
as representations of $O(n)$.

If $n=1$, the definition \eqref{eq:harmonic} of $\pi_a^{O(1)}$ is
still reasonable. Then $\pi_a^{O(1)}$ is one-dimensional if $a=0$ or
$1$, and zero for $a\ge 2$. The formula \eqref{eq:Osphere} is still
valid.

If $n=2$, the definition \eqref{eq:harmonic} of $\pi_a^{O(2)}$ is
still reasonable, and \eqref{eq:Osphere} is still valid. Then
$\pi_a^{O(2)}$ is one-dimensional if $a=0$, and two-dimensional for
$a\ge 1$.
\end{subequations}

\begin{subequations}\label{se:Rorbit}
Here is the orbit method perspective. The Lie algebra ${\mathfrak
  g}_0$ consists of $n\times n$ skew-symmetric matrices; ${\mathfrak
  h}_0$ is the subalgebra in which the first row and column are
zero. We can identify ${\mathfrak g}_0^*$ with ${\mathfrak g}_0$ using
the invariant bilinear form
$$B(X,Y) = \tr(XY).$$
Doing that, define
\begin{equation}
  a_{\text{orbit}} = a+(n-2)/2
\end{equation}

{\small
\begin{equation}
\lambda(a_{\text{orbit}}) =\begin{pmatrix} 0& a_{\text{orbit}}/2 &
\quad 0&\dots & 0\\ -a_{\text{orbit}}/2 & 0& \quad 0 &\dots &0\\[1ex]
0&0 &\\
\vdots&\vdots & & \text{\Large $0_{(n-2)\times (n-2)}$}\\[-.5ex]
 \\ 0&  0 &&&\end{pmatrix} \in ({\mathfrak g}_0/{\mathfrak h}_0)^*.
\end{equation}}

\smallskip
The isotropy group for $\lambda(a_{\text{orbit}})$ is
\begin{equation}
  O(n)^{\lambda(a_{\text{orbit}})} = SO(2)\times O(n-2)
  =_{\text{def}} L.
\end{equation}

  \smallskip
With this notation,
\begin{equation}
  \pi_a^{O(n)} = \pi(\text{orbit\ } \lambda(a_{\text{orbit}})).
\end{equation}
The reason this is true is that the infinitesimal character of the
orbit method representation on the right is (by \eqref{eq:cohindD})
\begin{equation}\begin{aligned}
  \lambda(a_{\text{orbit}}) -\rho_L &=
  (a+(n-2)/2,-(n-4)/2,-(n-6)/2,\cdots)\\
  &= \text{infinitesimal character of\ } \pi_a^{O(n)}.\end{aligned}
\end{equation}

An aspect of the orbit method perspective is that the ``natural''
dominance condition is no longer $a\ge 0$ but rather
\begin{equation}
  a_{\text{orbit}} > 0 \iff a > -(n-2)/2.
\end{equation}
For the compact group $O(n)$ we have
\begin{equation}
\pi(\text{orbit\ } \lambda(a_{\text{orbit}}))=0, \qquad 0 > a >
  -(n-2)/2,
\end{equation}
(for example because the infinitesimal characters of these
representations are singular) so the difference is not important. But
matters will be more interesting in the noncompact case (Section
\ref{sec:Opq}).
\end{subequations} 

Back in the general world of a homogeneous space $G/H$ for compact
groups, fix a (positive) $G$-invariant metric on ${\mathfrak g}_0 =
\Lie(G)$, and write
\begin{equation} \Omega_G = -\text{(sum of squares of an orthonormal
    basis)}.
 \end{equation}
for the corresponding Casimir operator. (We use a minus sign because
natural choices for the metric are negative definite rather than
positive definite.) The $G$-invariant metric on ${\mathfrak g}_0$
defines an $H$-invariant metric on
${\mathfrak g}_0/{\mathfrak h}_0 \simeq T_e(G/H)$, and therefore a
$G$-invariant Riemannian structure on $G/H$. Write
\begin{equation}
  L = \text{negative of Laplace-Beltrami operator on $G/H$,}
\end{equation}
a $G$-invariant differential operator. According to
\cite{GGA}*{Exercise II.A4}, the action of $\Omega_G$ on functions on
$G/H$ is equal to the action of $L$. (The Exercise is stated for symmetric
spaces, but the proof on page 568 works in the present setting.) Consequently

\medskip
\centerline{on an irreducible $G$-representation $\pi \subset
    C^\infty(G/H)$,}
\centerline{$L$ acts by the scalar $\pi(\Omega_G)$.}

\medskip
So we need to be able to calculate these scalars. If $T$ is a maximal
torus in $G$, and $\pi$ has highest weight $\lambda\in {\mathfrak
  t}^*$, then
\begin{equation}
  \pi(\Omega_G) = \langle \lambda + 2\rho,\lambda \rangle = \langle
  \lambda+\rho, \lambda+ \rho\rangle - \langle \rho,\rho\rangle.
\end{equation}
Here $2\rho\in {\mathfrak t}^*$ is the sum of the positive roots. (The
second formula relates this scalar to the infinitesimal character
written in \eqref{eq:Oinfchar} above.)

\begin{subequations}\label{se:Rspec}
  Now we're ready to calculate the spectrum of the spherical Laplace
  operator $L$. We need to calculate
  $\pi^{O(n)}_a(\Omega_{O(n)})$. The sum of the positive roots is
  \begin{equation}\label{eq:Orho}
    2\rho(O(n)) = (n-2, n-4,\cdots,n-2[n/2]).
  \end{equation}
  (Recall that we have identified weights of $T=SO(2)^{[n/2]}$ with
  $[n/2]$-tuples of integers.) Because our highest weight is
  \begin{equation}
    \lambda = (a,0,\ldots,0),
  \end{equation}
  we find
  \begin{equation}\label{eq:LR}
    \pi^{O(n)}_a(\Omega_{O(n)}) = a^2 + (n-2)a = a_{\text{orbit}}^2 - (n-2)^2/4.
  \end{equation}
\end{subequations} 
\begin{theorem} \label{thm:Rspec} Suppose $n\ge 3$. The eigenvalues of
  the (negative) 
  Laplace-Beltrami operator $L$ on
$S^{n-1}$ are $a^2 + (n-2)a$, for all non-negative integers $a$. The
multiplicity of this eigenvalue is
$$ \frac{(a+n/2 -1)\prod_{j=1}^{n-3} (a+j)}{(n/2 -1)\cdot (n-3)!},$$
a polynomial in $a$ of degree $n-2$.
\end{theorem}
In Sections \ref{sec:C}--\ref{sec:O} we'll repeat this calculation using other
groups.

\section{The complex calculation}
\label{sec:C}
\setcounter{equation}{0}

\begin{subequations}\label{se:Usphere}
Suppose $n\ge 1$ is an integer. Write $U(n)$ for the unitary group
of the standard Hermitian inner product on ${\mathbb C}^n$, and
\begin{equation}
  S^{2n-1} = \{v\in {\mathbb C}^n \mid \langle v,v\rangle = 1\}
\end{equation}
for the $(2n-1)$-dimensional sphere. We choose as a base point
\begin{equation}
  e_1 = (1,0,\ldots,0) \in S^{2n-1},
\end{equation}
which makes sense by our assumption that $n\ge 1$.
Then $U(n)$ acts transitively on $S^{2n-1}$, and the isotropy group at
$e_1$ is
\begin{equation}
  U(n)^{e_1} \simeq U(n-1);
\end{equation}
we embed $U(n-1)$ in $U(n)$ by acting on the last $n-1$
coordinates. This shows
\begin{equation}
  S^{2n-1} \simeq U(n)/U(n-1).
\end{equation}
\end{subequations} 

\begin{subequations}\label{se:Ureps}
Here is the representation theory. We omit the case $n=1$,
which is a degenerate version of the same thing; so assume $n\ge 2$. A
maximal torus in
$U(n)$ is
\begin{equation}
  T = U(1)^n
\end{equation}
so a weight is an $n$-tuple of integers. For all integers $b\ge 0$ and
$c\ge 0$ there is an irreducible representation $\pi^{U(n)}_{b,c}(n)$ of highest
weight
\begin{equation}
  (b,0,\ldots,0,-c), \qquad \dim \pi^{U(n)}_{b,c}= \frac{(b+c+n-1)\prod_{j=1}^{n-2}
    (b+j) 
    (c+j)}{(n -1)\cdot [(n-2)!]^2}. 
\end{equation}
Notice that the polynomial giving the dimension has degree $2n-3$ in
the variables $b$ and $c$. A
natural description of the representation is
\begin{equation}
  \pi^{U(n)}_{b,c} \simeq S^b({\mathbb C}^n) \otimes S^c(\overline{\mathbb
    C}^n)/r^2 S^{b-1}({\mathbb C}^n) \otimes S^{c-1}(\overline{\mathbb
    C}^n);
\end{equation}
what we divide by is zero if $b$ or $c$ is zero. The space is (a
quotient of) polynomial functions on ${\mathbb C}^n$, homogeneous of
degree $b$ in the holomorphic coordinates and homogeneous of degree
$c$ in the antiholomorphic coordinates.

Using the calculation of $\rho$ given in \eqref{eq:Urho} below, we
find
\begin{small}\begin{equation}\label{eq:Uinflchar}
\hbox{infl.~char.}\left(\pi^{U(n)}_{b,c}\right) = (b+(n-1)/2,(n-3)/2,\cdots,
-(n-3)/2,-(c+(n-1)/2).
\end{equation}\end{small}

The key fact (again in the notation of \eqref{se:invts}) is that
\begin{equation}\label{e:Ckey}
  \dim [\pi^{U(n)}_{b,c}]^{U(n-1)} = 1 \quad (b\ge 0, \ c\ge 0),
  \qquad \dim\pi^{U(n-1)} = 0 \quad (\pi \not\simeq \pi^{U(n)}_{b,c}).
\end{equation}
Therefore
\begin{equation}
  L^2(S^{2n-1}) \simeq \sum_{b\ge 0, \ c\ge 0} \pi^{U(n)}_{b,c}
\end{equation}
as representations of $U(n)$.

We add one more piece of representation-theoretic information,
without explaining yet why it is useful. If we write $U(1)$ for the
multiplication by unit scalars in the first coordinate, then $U(1)$ commutes
with $U(n-1)$. In any representation of $U(n)$, $U(1)$ therefore
preserves the $U(n-1)$-fixed vectors. The last fact is
\begin{equation}\label{e:extraCreps}
  \text{$U(1)$ acts on $[\pi^{U(n)}_{b,c}]^{U(n-1)}$ by the weight $b-c$.}
\end{equation}
\end{subequations}

\begin{subequations}\label{se:Corbit}
Here is the orbit method perspective. The Lie algebra ${\mathfrak
  g}_0$ consists of $n\times n$ skew-hermitian matrices; ${\mathfrak
  h}_0$ is the subalgebra in which the last row and column are zero.
We can identify ${\mathfrak
  g}_0^*$ with ${\mathfrak g}_0$ using the invariant bilinear form
$$B(X,Y) = \tr(XY).$$
Doing that, define
\begin{equation}
  b_{\text{orbit}} = b+(n-1)/2, \qquad c_{\text{orbit}} = c+ (n-1)/2.
\end{equation}
We need also an auxiliary parameter
\begin{equation}
  r_{\text{orbit}} = (b_{\text{orbit}}c_{\text{orbit}})^{1/2}.
\end{equation}

Now define a linear functional

\smallskip
{\small
\begin{equation}
\lambda(b_{\text{orbit}},c_{\text{orbit}}) =\begin{pmatrix}  i(b_{\text{orbit}} -
c_{\text{orbit}})& r_{\text{orbit}} & \quad 0&\dots & 0\\ -r_{\text{orbit}}
& 0 & \quad 0 &\dots &0\\[1ex] 0&0 &\\ \vdots&\vdots & & \text{\Large
  $0_{(n-2)\times (n-2)}$}\\[-.5ex]  \\ 0&  0 &&&\end{pmatrix} \in
({\mathfrak g}_0/{\mathfrak h}_0)^*.
\end{equation}}

\smallskip
This skew-hermitian matrix has been constructed to be orthogonal to
${\mathfrak h}_0$, and to have eigenvalues $ib_{\text{orbit}}$,
$-ic_{\text{orbit}}$, and $n-2$ zeros. Its isotropy group is (as long
as $r_{\text{orbit}} \ne 0$)
\begin{equation}
U(n)^{\lambda(b_{\text{orbit}},c_{\text{orbit}})} = U(1)
\times U(n-2) \times U(1) =_{\text{def}} L;
\end{equation}
the first and last $U(1)$ factors are not the usual ``coordinate''
$U(1)$ factors, but rather correspond to the $ib_{\text{orbit}}$ and
$-ic_{\text{orbit}}$ eigenspaces respectively.
\smallskip
With this notation,
\begin{equation}
  \pi_{b,c}^{U(n)} = \pi(\text{orbit\ } \lambda(b_{\text{orbit}},c_{\text{orbit}})).
\end{equation}
An aspect of the orbit method perspective is that the ``natural''
dominance condition is no longer $b,c\ge 0$ but rather
\begin{equation}
  b_{\text{orbit}} > 0 \iff b > -(n-1)/2, \qquad c_{\text{orbit}} > 0
  \iff c > -(n-1)/2.
\end{equation}
For the compact group $U(n)$ we have
\begin{equation}\begin{aligned}
\pi(\text{orbit\ } \lambda(b_{\text{orbit}},c_{\text{orbit}}))=0
\quad &\text{if}\quad 0 > b > -(n-1)/2 \\
&\text{or}\quad 0 > c > -(n-1)/2,\end{aligned}
\end{equation}
so the difference is not important. But matters will be more
interesting in the noncompact case (Section \ref{sec:Upq}).
\end{subequations} 

\begin{subequations}\label{se:Cspec}
  Now we're ready for spectral theory.   We need to calculate
  $\pi^{U(n)}_{b,c}(\Omega_{U(n)})$. The sum of the positive roots is
  \begin{equation}\label{eq:Urho}
    2\rho(U(n)) = (n-1, n-3,\cdots,-(n-1)).
  \end{equation}
  (Recall that we have identified weights of $T=U(1)^n$ with
  $n$-tuples of integers.) Because our highest weight is
  \begin{equation}
    \lambda = (b,0,\ldots,-c),
  \end{equation}
  we find
  \begin{equation}\label{eq:LC}\begin{aligned}
      \pi^{U(n)}_{b,c}(\Omega_{U(n)}) &= b^2 + c^2 + (n-1)(b+c)\\
      &= b_{\text{orbit}}^2 + c_{\text{orbit}}^2 - (n-1)^2/2.
      \end{aligned}
  \end{equation}

Just as for the representation theory above, we'll add one more piece
of information without explaining why it will be useful:
\begin{equation}
  [\pi^{U(n)}_{b,c}]^{U(n-1)}(\Omega_{U(1)}) = (b-c)^2 = b^2 + c^2 - 2bc.
\end{equation}
Combining the last two equations gives
\begin{equation}\label{e:extraCspec}
  [\pi^{U(n)}_{b,c}]^{U(n-1)}(2\Omega_{U(n)} - \Omega_{U(1)}) = (b+c)^2 +
  (2n-2)(b+c).
\end{equation}

\end{subequations} 
\begin{theorem} Suppose $n\ge 2$. The eigenvalues of the (negative)
  Laplace-Beltrami operator $L_U$ on 
$S^{2n-1}$ are $b^2 + c^2 + (n-1)(b+c)$, for all non-negative integers
  $b$ and $c$. The
  multiplicity of this eigenvalue is
$$ \frac{(b+c+n-1)\prod_{j=1}^{n-2}
    (b+j) \prod_{k=1}^{n-2} (c+k)}{(n -1)\cdot (n-2)! \cdot (n-2)!}$$
  a polynomial in $b$ and $c$ of total degree $2n-3$.

A little more precisely, the multiplicity of an eigenvalue $\lambda$
is the sum over all expressions
$$\lambda = b^2 + c^2 + (n-1)(b+c)$$
(with $b$ and $c$ nonnegative integers) of the indicated polynomial in
$b$ and $c$.
\end{theorem}

Let us compute the first few eigenvalues when $n=2$, so that we are
looking at $S^3$. Some numbers are in Table \ref{table:CLaplacian}. We
have also included eigenvalues and multiplicities from the
calculation with $O(4)$ acting on $S^3$, and the peculiar added
calculations from \eqref{e:extraCreps} and \eqref{e:extraCspec}.

\begin{table}\label{table:CLaplacian}
  \caption{\bf Casimir eigenvalues and multiplicities on
    $S^3$}

 \smallskip
  \begin{tabular}{|c|c|c|c|c || c | c | c|}
    \hline
    $b$ & $c$ & \Tstrut\Bstrut\small{$\pi^{U(n)}_{b,c}(\Omega_{U(2)})$}
    &{\scriptsize $[\pi^{U(n)}_{b,c}]^{U(1)}(2\Omega_{U(2)} - \Omega_{U(1)})$}
    & dim &  $a$
    & \small{$\pi^{O(4)}_a(\Omega_{O(4)})$} & dim \\[1ex]
    \hline
    0 & 0 & 0 & 0 & 1 & 0 & 0 & 1\\
\hline
    0 & 1 & 2 & 3 & 2 & 1 & 3 & 4\\
    1 & 0 & 2 & 3 & 2 &&&\\
\hline
    0 & 2 & 6 & 8 & 3 & 2 & 8 & 9\\
    1 & 1 & 4 & 8 & 3&&&\\
    2 & 0 & 6 & 8 & 3&&&\\
    \hline
    0 & 3 & 12 & 15 & 4 & 3 & 15 & 16\\
    1 & 2 &  8 & 15 & 4 &&&\\
    2 & 1 &  8 & 15 & 4 &&&\\
    3 & 0 & 12 & 15 & 4 &&&\\
    \hline
  \end{tabular}
\end{table}
Since each half (left and the right) of the table concerns $S^3$,
there should be some relationship between them. There are indeed
relationships, but they are not nearly as close as one might
expect. What is being calculated in each case is the spectrum of a
Laplace-Beltrami operator. It is rather clear that the spectra are
quite different: the multiplicities calculated with $U(2)$ are
smaller than the multiplicities calculated with $O(4)$, and the actual
eigenvalues are smaller for $U(2)$ as well.

The reason for this is that metric $g_{O}$ that we
used in the $O(2n)$ calculation is not the same as the metric $g_{U}$
that we used in the $U(n)$ calculation.  There are two aspects to the
difference. Recall that
\begin{equation}\label{eq:Rtan}
  T_{e_1}(S^{n-1}) = \{(0,v_2,\cdots,v_{n}) \mid v_j \in {\mathbb
    R}\} \simeq {\mathbb R}^{n-1}.
\end{equation}
In this picture, we will see that $g_{O}$ is the usual inner product
on ${\mathbb R}^{n-1}$. In the $U(n)$ picture,
\begin{equation}\label{eq:T2n-1}
  T_{e_1}(S^{2n-1}) = \{(it_1,z_2,\cdots,z_{n}) \mid t\in {\mathbb
    R},\ z_j \in {\mathbb C}\} \simeq {\mathbb R} + {\mathbb C}^{n-1}.
\end{equation}
In this picture, $g_U$ is actually {\em twice} the usual inner product
on ${\mathbb C}^{n-1}$:
\begin{equation}\label{eq:gOU}
  |(0,x_2+iy_2,\cdots,x_n+iy_n)|^2_{g_U} =
  2|(0,0,x_2,y_2,\cdots,x_n,y_n)|^2_{g_O}.
\end{equation}
Here is how to see this factor of two. The
Riemannian structure $g_O$ for $O(n)$ is
related to the invariant bilinear form on ${\mathfrak o}(n)$
\begin{equation}\label{eq:Oform}
  \langle X,Y\rangle_{O(n)} = (1/2)\tr(XY).
  \end{equation}
The reason for the
factor of $1/2$ is so that the form restricts to (minus) the ``standard''
inner product on the Cartan subalgebra ${\mathfrak s}{\mathfrak
  o}(2)^{[n/2]} \simeq {\mathbb R}^{[n/2]}$. Now suppose that
$$v\in {\mathbb R}^{n-1} \simeq T_{e_1}(S^{n-1}).$$
The tangent vector $v$ is given by the $n\times n$ skew-symmetric
matrix $A(v)$ with first row $(0,v)$, first column $(0,-v)^t$, and all
other entries zero. Then
\begin{equation}\label{eq:gO}
  |v|^2_{g_O} = -\langle A(v),A(v)\rangle_{O(n)} = -(1/2)(\tr(A(v)A(v)))
  = |v|^2,
\end{equation}
proving the statement after \eqref{eq:Rtan} about $g_0$.

For similar reasons, $g_U$ is related
to the invariant form on ${\mathfrak u}(n)$
\begin{equation}\label{eq:Uform1}
\langle Z,W\rangle_{U(n)} = \RE\tr(ZW) = (1/2)(\tr(ZW) +
\overline{\tr(ZW)}).
\end{equation}
If $z\in {\mathbb C}^{n-1} \subset T_{e_1}(S^{2n-1})$, then the
tangent vector $z$ is given by the $n\times n$ skew-Hermitian matrix $B(z)$
with first row $(0,z)$, first column $(0,-\overline{z})^t$, and all
other entries zero. Therefore
\begin{equation}\label{eq:gU}
  |z|^2_{g_U} = -\langle B(z),B(z)\rangle_{U(n)} = -\RE(\tr(B(z)B(z)) =
  2|z|^2.
  \end{equation}
Now equations \eqref{eq:gO} and \eqref{eq:gU} prove \eqref{eq:gOU}

Doubling the Riemannian metric has the effect of
dividing the Laplace operator by two, and so dividing the eigenvalues
by two. For this reason, the eigenvalues computed using $U(n)$ ought
to be half of those computed using $O(2n)$.

But that is still not what the table says. The reason is that in the
$U(n)$ picture, there is a ``preferred'' line in each tangent space,
corresponding to the fibration
$$S^1 \rightarrow S^{2n-1} \rightarrow {\mathbb C}{\mathbb P}^{n-1}.$$
In our coordinates in \eqref{eq:T2n-1}, it is the coordinate
$t_1$. The skew-Hermitian matrix $C(it_1)$ involved has $it_1$ in the first
diagonal entry, and all other entries zero.
\begin{equation}
  |(it_1,0,\cdots,0)|^2_{g_U} = -\langle C(it_1),C(it_1)\rangle_{U(n)}
 = t_1^2 =  |(0,t_1,0,0,\cdots,0,0)|^2_{g_O}:
\end{equation}
no factor of two. So the metric attached to the $U(n)$ action is
fundamentally different from the metric attached to the $O(2n)$
action. In the $U(n)$ case, there is a new (non-elliptic) Laplacian $L_{U(1)}$
acting in the direction of the $S^1$ fibration only. The remarks about
metrics above say that
\begin{equation}\label{eq:LOU}
  L_O = 2L_U - L_{U(1)}.
\end{equation}
(The reason is that the sum of squares of derivatives in $L_O$ is
almost exactly twice the sum of squares $L_U$; except that this factor
of two is not needed in the direction of the $U(1)$ fibration.)
The ``extra'' calculations \eqref{e:extraCreps} and
\eqref{e:extraCspec} are calculating the spectrum of
$L_{U(1)}$ representation-theoretically; so the column
$$[\pi^{U(n)}_{b,c}]^{U(1)}(2\Omega_{U(2)} - \Omega_{U(1)})$$
in the table above is calculating the spectrum of the classical Laplacian
$L_O$.

Here is a final representation-theoretic statement, explaining how the
$U(n)$ and $O(2n)$ calculations fit together.
\begin{theorem} \label{thm:OUcptbranch} Suppose $n\ge 2$, and $a$ is a
  non-negative integer. Using the inclusion $U(n)\subset O(2n)$, we have
  $$\pi^{O(2n)}_a|_{U(n)} = \sum_{\substack{0\le b,c \\[.1ex] b+c = a}} \pi^{U(n)}_{b,c}.$$
The contribution of these representations to the spectrum of the
$O(2n)$-invariant Laplacian $L_O$ is
$$\begin{aligned}\pi^{O(2n)}_a(\Omega_{O(2n)}) &= a^2 + (2n-2)a \\
  & = (b+c)^2 + 2(n-1)(b+c) \\
  &= [\pi^{U(n)}_{b,c}]^{U(n-1)}(2\Omega_{U(n)} - \Omega_{U(1)}).\end{aligned}$$
\end{theorem}

\section{The quaternionic calculation}
\label{sec:H}
\setcounter{equation}{0}

\begin{subequations}\label{se:Hsphere}
Suppose $n\ge 1$ is an integer. Write $Sp(n)$ for the unitary group
of the standard Hermitian inner product on ${\mathbb H}^n$. This is a
group of ${\mathbb H}$-linear transformations; that is, ${\mathbb
  R}$-linear transformations commuting with scalar multiplication by
${\mathbb H}$. Because ${\mathbb H}$ is noncommutative, these scalar
multiplications do {\em not} commute with each other, and so are {\em
  not} linear. It is therefore possible and convenient to enlarge
$Sp(n)$ to
\begin{equation}\label{eq:Spbig}
  Sp(n) \times Sp(1) = Sp(n)_{\text{linear}} \times
  Sp(1)_{\text{scalar}};
\end{equation}
the second factor is scalar multiplication by unit quaternions. This
enlarged group acts on ${\mathbb H}^n$, by the
formula
\begin{equation}
  (g_{\text{linear}},z_{\text{scalar}})\cdot v = gvz^{-1};
\end{equation}
we need the inverse to make the right action of scalar multiplication
into a left action. The action preserves length, and so can be
restricted to the $(4n-1)$-dimensional sphere
\begin{equation}
  S^{4n-1} = \{v\in {\mathbb H}^n \mid \langle v,v\rangle = 1\}
\end{equation}
We choose as a base point
\begin{equation}
  e_1 = (1,0,\ldots,0) \in S^{4n-1},
\end{equation}
which makes sense by our assumption that $n\ge 1$.
Then $Sp(n)\times Sp(1)$ acts transitively on $S^{2n-1}$, and the isotropy group at
$e_1$ is
\begin{equation}
  [Sp(n)\times Sp(1)]^{e_1} \simeq Sp(n-1) \times Sp(1)_\Delta.
\end{equation}
Here we embed $Sp(n-1)$ in $Sp(n)$ by acting on the last $n-1$
coordinates, and the last factor is the diagonal subgroup in
$Sp(1)_{\text{linear}}$ (acting on the first coordinate) and
$Sp(1)_{\text{scalar}}$. This shows
\begin{equation}
  S^{4n-1} \simeq [Sp(n)\times Sp(1)]/[Sp(n-1) \times Sp(1)_\Delta].
\end{equation}
\end{subequations} 

\begin{subequations}\label{se:Hreps}
Here is the representation theory. We omit the case $n=1$,
which is a degenerate version of the same thing; so assume $n\ge 2$. A
maximal torus in
$Sp(n)$ is
\begin{equation}
  T = U(1)^n,
\end{equation}
$n$ copies of the unit complex numbers acting diagonally on ${\mathbb
  H}^n$. A weight is therefore an $n$-tuple of integers.  For all
integers $d\ge e\ge 0$ there is an irreducible representation
\begin{equation}\begin{aligned}
    \pi^{Sp(n)}_{d,e} &\text{\ of highest weight\ } (d,e,0,\ldots,0,0), \\[.3ex]
    \dim \pi^{Sp(n)}_{d,e}&= \frac{(d+e+2n-1)(d-e+1)\prod_{j=1}^{2n-3}
    (d+j+1)(e+j)}{(2n -1)(2n-2)\cdot [(2n-3)!]^2}. 
\end{aligned}\end{equation}

A maximal torus in $Sp(1)$ is $U(1)$, and a weight is an integer. For
each integer $f\ge 0$ there is an irreducible representation
\begin{equation}
  \pi^{Sp(1)}_f \text{\ of highest weight\ } f,\quad \dim \pi^{Sp(1)}_f =
  f+1.
  \end{equation}

We are interested in the representations (for $d\ge e \ge 0$)
\begin{small}
\begin{equation}\begin{aligned}
    \pi^{Sp(n)\times Sp(1)}_{d,e} &= \pi^{Sp(n)}_{d,e}\otimes
    \pi^{Sp(1)}_{d-e} \\
    \dim \pi^{Sp(n)\times Sp(1)}_{d,e}&= \frac{(d+e+2n-1)(d-e+1)^2\prod_{j=1}^{2n-3}
    (d+j+1)(e+j)}{(2n -1)(2n-2)\cdot [(2n-3)!]^2}. 
\end{aligned}\end{equation}\end{small}
Notice that the polynomial giving the dimension has degree $4n-3$.

Using the calculation of $\rho$ given in \eqref{eq:Hrho} below, we
find
\begin{small}\begin{equation}\label{eq:Hinflchar}
\hbox{infl.~char.}(\pi^{Sp(n)\times Sp(1)}_{d,e}) =
(d+n,e+(n-1),n-2,\cdots,1)(d-e+1).
\end{equation}\end{small}

The key fact is that
\begin{equation}\label{e:Hkey}\begin{aligned}
  \dim [\pi^{Sp(n)\times Sp(1)}_{d,e}]^{Sp(n-1)\times Sp(1)_\Delta} &= 1 \qquad (d\ge e \ge 0),\\
  \dim\pi^{Sp(n-1)\times Sp(1)_\Delta} &= 0 \quad (\pi \not\simeq
  \pi^{Sp(n)\times Sp(1)}_{d,e}).
  \end{aligned}
\end{equation}
Therefore
\begin{equation}
  L^2(S^{4n-1}) \simeq \sum_{d\ge e \ge 0} \pi^{Sp(n)\times Sp(1)}_{d,e}
\end{equation}
as representations of $Sp(n)\times Sp(1)$. 

Here is one more piece of representation-theoretic information. We saw that
$Sp(n-1)\times Sp(1)_\Delta \subset Sp(n-1)\times Sp(1)\times Sp(1)
\subset Sp(n) \times Sp(1)$; so inside any
representation of $Sp(n)\times Sp(1)$ we get a natural representation of
$Sp(1)\times Sp(1)$ generated by the $Sp(n-1)\times Sp(1)_\Delta$
fixed vectors. The last fact is
\begin{equation}\label{e:extraHreps}\begin{aligned}
\ [Sp(1)\times Sp(1)] &\cdot [\pi^{Sp(n)\times
    Sp(1)}_{d,e}]^{Sp(n-1)\times Sp(1)_\Delta} \\ &= \text{irr of
  highest weight  $(d-e,d-e)$.}
\end{aligned}\end{equation} 
This representation has infinitesimal character
\begin{equation}\label{eq:Hsubinflchar}\begin{aligned}
\hbox{infl.~char.}\big([Sp(1)\times Sp(1)]&\cdot[\pi^{Sp(n)\times
    Sp(1)}_{d,e}]^{Sp(n-1)\times Sp(1)_\Delta}\big) \\ &= (d-e+1,d-e+1).
\end{aligned}\end{equation}

\end{subequations}

\begin{subequations}\label{se:Horbit}
Here is the orbit method perspective. (To simplify the notation, we
will discuss only $G=Sp(n)$ rather than $Sp(n) \times Sp(1)$.) The Lie
algebra ${\mathfrak 
  g}_0$ consists of $n\times n$ skew-hermitian quaternionic matrices;
${\mathfrak h}_0$ is the subalgebra in which the last row and column are zero.
Define
\begin{equation}
  d_{\text{orbit}} = d+(n-1), \qquad e_{\text{orbit}} = e+ (n-2).
\end{equation}
We need also an auxiliary parameter
\begin{equation}
  r_{\text{orbit}} = (d_{\text{orbit}}e_{\text{orbit}})^{1/2}.
\end{equation}

Now define a linear functional

\smallskip
{\small
\begin{equation}
\lambda(d_{\text{orbit}},e_{\text{orbit}}) =\begin{pmatrix}  i(d_{\text{orbit}} +
e_{\text{orbit}})& r_{\text{orbit}} & \quad 0&\dots & 0\\ -r_{\text{orbit}}
& 0 & \quad 0 &\dots &0\\[1ex] 0&0 &\\ \vdots&\vdots & & \text{\Large
  $0_{(n-2)\times (n-2)}$}\\[-1ex]  \\ 0&  0 &&&\end{pmatrix} \in
({\mathfrak g}_0/{\mathfrak h}_0)^*.
\end{equation}}

This skew-hermitian matrix has been constructed to be orthogonal to
${\mathfrak h}_0$, and to be conjugate by $G$ to
\begin{equation}
  \begin{pmatrix} id_{\text{orbit}}& 0 &0& \dots & 0\\
    0& ie_{\text{orbit}}& 0& \dots & 0\\[1ex] 0& 0 &&\\ \vdots& \vdots&
      & \text{\Large $0_{(n-2)\times(n-2)}$}\\
  0&0 \end{pmatrix}
\end{equation}

\smallskip
With this notation,
\begin{equation}
  \pi_{d,e}^{Sp(n)} = \pi(\text{orbit\ } \lambda(d_{\text{orbit}},e_{\text{orbit}})).
\end{equation}
An aspect of the orbit method perspective is that the ``natural''
dominance condition is no longer $d\ge e\ge 0$ but rather
\begin{equation}
  d_{\text{orbit}} > e_{\text{orbit}} > 0 \iff d + 1 >  e > -(n-1).
\end{equation}
For the compact group $Sp(n)$ we have
\begin{equation}
\pi(\text{orbit\ } \lambda(d_{\text{orbit}},e_{\text{orbit}}))=0
\quad \text{if}\quad 0 > e > -(n-1)
\end{equation}
so the difference is not important. But matters will be more
interesting in the noncompact case (Section \ref{sec:Sppq}).
\end{subequations} 

\begin{subequations}\label{se:Hspec}
  Now we're ready for spectral theory.   Because the group is a
  product, it is natural to calculate the eigenvalues of the Casimir
  operators from the two factors separately. We calculate first
  $\pi^{Sp(n)\times Sp(1))}_{d,e}(\Omega_{Sp(n)})$. The sum of the positive roots is
  \begin{equation}\label{eq:Hrho}
    2\rho(Sp(n)) = (2n, 2n-2,\cdots,2).
  \end{equation}
Because our highest weight for $Sp(n)$ is
  \begin{equation}
    \lambda = (d,e,0,\ldots,0),
  \end{equation}
  we find
  \begin{equation}\label{eq:LH}\begin{aligned}
      \pi^{Sp(n)}_{d,e}(\Omega_{Sp(n)}) &= d^2 + e^2 + 2nd + 2(n-1)e\\
      &=d_{\text{orbit}}^2 + e_{\text{orbit}}^2 - n^2 - (n-1)^2. \end{aligned}
  \end{equation}
Similarly
\begin{equation}
  \pi^{Sp(n)\times Sp(1)}_{d,e}(\Omega_{Sp(1)}) = (d-e)^2 + 2(d-e) = d^2 + e^2 -
  2de + 2(d-e).
\end{equation}
Combining the last two equations gives
\begin{equation}\label{e:extraHspec}\begin{aligned}
  \pi^{Sp(n)\times Sp(1)}_{d,e}(2\Omega_{Sp(n)} - \Omega_{Sp(1)}) &= (d+e)^2 +
  (4n-2)(d+e)\\ &=\pi^{O(4n)}_{d+e}(\Omega_{O(4n)}).
  \end{aligned}\end{equation}
This formula is first of all just an algebraic identity, obtained by
plugging in $a=d+e$ and $4n$ in the formula \eqref{eq:LR}. But it has
a more serious meaning. Let us directly compare the metrics $g_O$ and
$g_{Sp}$ on $S^{4n-1}$, as we did for $g_{U}$ in Section
\ref{sec:C}. We find that on a $(4n-4)$-dimensional subspace of the
tangent space, $g_O$ is some multiple $x\cdot g_{Sp}$; and on the
orthogonal $3$-dimensional subspace (corresponding to the $Sp(1)\simeq
S^3$ fibers of the bundle $S^{4n-1} \rightarrow {\mathbb
  P}^{n-1}({\mathbb H})$) there is a different relationship $g_O =
y\cdot g_{Sp}$. (It is not difficult to check by more careful
calculation that $x=2$ and $z=1$, but we are looking here for what is
obvious.)  It follows that
$$    L_{O} = xL_{Sp} - zL_{Sp(1)},$$
  exactly as in \eqref{eq:LOU}. If now
  $$\pi_{d,e}^{Sp(n)\times Sp(1)} \subset \pi_a^{O(4n)},$$
  then we conclude (by computing the Laplacian separately in these two
  representations) that there is (for all integers $d\ge e \ge 0$) an
  algebraic identity
$$    x(d^2+e^2+2nd +2(n-1)e) - z((d-e)^2 + 2(d-e)) = a^2 + (4n-2)a;$$
here $a\ge 0$ is some integer depending on $d$ and $e$. Since every
integer $a \ge 0$ must appear in such an identity, it follows easily
that $x=2$ and $z=1$, and that $a=d+e$. In particular,
  \begin{equation}\label{eq:LOSp}
    L_{O} = 2L_{Sp} - L_{Sp(1)}.
  \end{equation}
This means that the equation \eqref{e:extraHspec} is describing two
calculations of $L_O$, in the subrepresentation
\begin{equation}
  \pi_{d,e}^{Sp(n)\times Sp(1)} \subset \pi_{d+e}^{O(4n)}.
  \end{equation}
\end{subequations} 

Here is what we have proven about how the $Sp(n)$ and $O(4n)$
calculations fit together.
\begin{theorem} \label{thm:OSpcptbranch} Suppose $n\ge 2$, and $a$ is
  a non-negative integer. Using the map $Sp(n)\times
  Sp(1)\rightarrow O(4n)$, we have
  $$\pi^{O(4n)}_a|_{Sp(n)\times Sp(1)} = \sum_{\substack{d\ge e \ge 0 \\ d+e= a}}
  \pi^{Sp(n)\times Sp(1)}_{d,e}.$$
 
The contribution of these representations to the spectrum of the
$O(4n)$-invariant Laplacian $L_O$ is
$$\begin{aligned}\pi^{O(4n)}_a(\Omega_{O(4n)}) &= a^2 + (4n-2)a \\
  & = (d+e)^2 + (4n-2)(d+e) \\
  &= \pi^{Sp(n)\times Sp(1)}_{d,e}(2\Omega_{Sp(n)} - \Omega_{Sp(1)}).\end{aligned}$$
\end{theorem}

\section{The octonionic calculation}
\label{sec:O}
\setcounter{equation}{0}

We will make no explicit discussion of octonions, except to say that
$F_4$ is related; and that the non-associativity of octonions makes it
impossible to define a ``projective space'' except in octonionic
dimension one. That is why this example is not part of an infinite
family like the real, complex, and quaternionic ones.

\begin{subequations}\label{se:Osphere}
Write $\Spin(9)$ for the compact spin double cover of $SO(9)$. This
group can be defined using a spin representation $\sigma$, which has dimension
$2^{(9-1)/2} = 16$. The representation is real, so we fix a
realization $(\sigma_{\mathbb R},V_{\mathbb R})$ on a
sixteen-dimensional real vector space. Of course the compact group
$\Spin(9)$ preserves a positive definite inner product on $V_{\mathbb
  R}$, and
\begin{equation}
  S^{15} = \{v\in V_{\mathbb R} \mid \langle v,v\rangle = 1\}
\end{equation}
We choose as a base point
\begin{equation}
  v_1 \in S^{15},
\end{equation}
Then $\Spin(9)$ acts transitively on $S^{15}$. (Once one knows that
$F_{4,c}/\Spin(9)$ is a (sixteen-dimensional) rank one Riemannian
symmetric space, and that the action of $\Spin(9)$ on the tangent
space at the base point is the spin representation, then this is
Cartan's result \eqref{e:rank1}.) The isotropy group at $v_1$ is
\begin{equation}
  \Spin(9)^{v_1} \simeq \Spin(7)'.
\end{equation}
The embedding of $\Spin(9)^{v_1}$ in $\Spin(9)$ can be described as follows.
First, we write
\begin{equation}
  \Spin(8)\subset \Spin(9)
\end{equation}
for the double cover of $SO(8) \subset SO(9)$. Next, we embed
\begin{equation}
  \Spin(7)' \ {\buildrel{\text{spin}}\over \longrightarrow}\  \Spin(8).
\end{equation}
(We use the prime to distinguish this subgroup from the double cover
of $SO(7) \subset SO(8)$, which we will call $\Spin(7) \subset
\Spin(8)$.) The way this works is that the spin representation of
$\Spin'(7)$ has dimension $2^{(7-1)/2}= 8$, is real, and preserves a
quadratic form, so $\Spin'(7) \subset SO(8)$. (Another explanation
appears in \eqref{se:bigG2sphere} below.) Now take the double cover
of this inclusion. This shows
\begin{equation}
  S^{15} \simeq \Spin(9)/\Spin(7)'.
\end{equation}
\end{subequations} 

\begin{subequations}\label{se:Oreps}
Here is the representation theory. A maximal torus in $\Spin(9)$ is a
double cover of $SO(2)^4 \subset SO(9)$. A weight is {\em either} a
$4$-tuple of integers (the weights factoring to $SO(2)^4$) {\em or} a
$4$-tuple from ${\mathbb Z} + 1/2$.
For all integers $x\ge 0$ and $y \ge 0$ there is an irreducible representation
\small
\begin{equation}\begin{aligned}
    \pi^{\Spin(9)}_{x,y} &\text{\ of highest weight\ } (y/2 +
    x,y/2,y/2,y/2),\\[.3ex]
 \dim \pi^{\Spin(9)}_{x,y} &=
  \frac{(2x+y+7)\prod_{j=1}^3 
    (x+j)(y+j+1)(y+2j-1)(x+y+j+3)}{7!\cdot 6! \cdot(1/2)}
\end{aligned}\end{equation}
\normalsize
Notice that the polynomial giving the dimension has degree $13$.

Using the calculation of $\rho$ given in \eqref{eq:octrho} below, we
find
\begin{small}
\begin{equation}\label{eq:octinflchar}
\hbox{infl.~char.}(\pi^{\Spin(9)}_{x,y}) =
((2x+y+7)/2,(y+5)/2,(y+3)/2,(y+1)/2).
\end{equation}
\end{small}

The key fact is that
\begin{equation}\label{e:Okey}\begin{aligned}
  \dim [\pi^{\Spin(9)}_{x,y}]^{\Spin(7)'} &= 1 \qquad (x\ge 0,\ y \ge
  0),\\
  \dim\pi^{\Spin(7)'} &= 0 \quad (\pi \not\simeq \pi^{\Spin(9)}_{x,y}).
\end{aligned}\end{equation}
Therefore
\begin{equation}
  L^2(S^{15}) \simeq \sum_{x\ge 0,\ y \ge 0} \pi^{\Spin(9)}_{x,y}
\end{equation}
as representations of $\Spin(9)$.

Here is one more piece of representation-theoretic information. We saw that
$\Spin(7)'\subset\Spin(8) \subset \Spin(9)$; so inside any
representation of $\Spin(9)$ we get a natural representation of
$\Spin(8)$ generated by the $\Spin(7)'$ fixed vectors. The last fact is
\begin{small}\begin{equation}\label{e:extraOreps}
\Spin(8)\cdot [\pi^{\Spin(9)}_{x,y}]^{\Spin(7)'} = \text{irr of
    highest weight  $(y/2,y/2,y/2,y/2)$.}
\end{equation}\end{small}
This representation has infinitesimal character
\begin{small}\begin{equation}\label{eq:octsubinflchar}
\hbox{infl.~char.}\left(\Spin(8)\cdot[\pi^{\Spin(9)}_{x,y}]^{\Spin(7)'}\right) =
((y+6)/2,(y+4)/2,(y+2)/2,y/2).
\end{equation}\end{small}
Here is why this is true. Helgason's theorem about symmetric spaces says that
the representations of $\Spin(8)$ of highest weights
\begin{equation}\label{eq:87'}
  (y/2,y/2,y/2,y/2)
\end{equation}
are precisely the ones having a $\Spin(7)'$-fixed
vector, and furthermore this fixed vector is unique.  The
corresponding statement for $\Spin(8)/\Spin(7)$ is the case $n=8$ of
Theorem \ref{thm:Rspec}. In that case the highest weights for
$\Spin(8)$ appearing are the multiples of the fundamental weight
$(1,0,0,0)$ (corresponding to the simple root at the end of the
``long'' leg of the Dynkin diagram of $D_4$. For
$\Spin(8)/\Spin(7)'$, the weights appearing must therefore be
multiples of the fundamental weight $(1/2,1/2,1/2,1/2)$ for a simple
root on one of the ``short'' legs of the Dynkin diagram, proving
\eqref{eq:87'}.

To complete the proof of \eqref{e:Okey} using \eqref{eq:87'} we need only the
classical branching theorem for $\Spin(8) \subset \Spin(9)$ (see for example
\cite{KBeyond}*{Theorem 9.16}).
\end{subequations} 

\begin{subequations}\label{se:Oorbit}
Here is the orbit method perspective. Define
\begin{equation}
  x_{\text{orbit}} = x+ 2, \qquad y_{\text{orbit}} = y+3.
\end{equation}
Then it turns out that there is a $9\times 9$ real skew-symmetric
matrix $\lambda(x_{\text{orbit}}, y_{\text{orbit}})$ (which we will
not attempt to write down) with the properties

\smallskip
\begin{equation} \begin{aligned}
\lambda(x_{\text{orbit}},y_{\text{orbit}}) &\in ({\mathfrak
  g}_0/{\mathfrak h}_0)^*\\
\lambda(x_{\text{orbit}},y_{\text{orbit}}) &\ \text{has eigenvalues}\\
  &\ \text{$\pm i(x_{\text{orbit}}/2+y_{\text{orbit}}/4)$ and
  $\pm i(y_{\text{orbit}}/4)$ (three times).} \end{aligned}
\end{equation}

\smallskip

Consequently
\begin{equation}
  \pi_{x,y}^{\Spin(9)} = \pi(\text{orbit\ }
  \lambda(x_{\text{orbit}},y_{\text{orbit}})).
\end{equation}
An aspect of the orbit method perspective is that the ``natural''
dominance condition is no longer $x,y \ge 0$ and but rather
\begin{equation}
  x_{\text{orbit}},\  y_{\text{orbit}} > 0 \iff x > -2, y > -3.
\end{equation}
For the compact group $\Spin(9)$ we have
\begin{equation}
\pi(\text{orbit\ } \lambda(x_{\text{orbit}},y_{\text{orbit}}))=0
\quad \text{if}\quad 0 > x > -2 \ \text{or} \ 0 > y > -3;
\end{equation}
so the difference is not important. But matters will be more
interesting in the noncompact case (Section \ref{sec:spinp9-p}).
\end{subequations} 

\begin{subequations}\label{se:Ospec}
  Now we're ready for spectral theory.   We need to calculate
  $\pi^{\Spin(9)}_{x,y}(\Omega_{\Spin(9)})$. The sum of the positive roots is
  \begin{equation}\label{eq:octrho}
    2\rho(\Spin(9)) = (7,5,3,1).
  \end{equation}
Because our highest weight is
  \begin{equation}
    \lambda = (y/2 + x,y/2,y/2,y/2),
  \end{equation}
  we find
  \begin{equation}\label{eq:LO}
    \pi^{\Spin(9)}_{x,y}(\Omega_{\Spin(9)}) = x^2 + y^2 + xy +8y + 7x.
  \end{equation}

Just as for the representation theory above, we'll add one more piece
of information without explaining why it will be useful:
\begin{equation}
  \left(\Spin(8) \cdot [\pi^{\Spin(9)}_{x,y}]^{\Spin(7)'}
  \right)(\Omega_{\Spin(8)}) =  y^2 + 6y \end{equation}
Combining the last two equations gives
\begin{equation}\label{e:extraOspec}\begin{aligned}
  \left(\Spin(8)\cdot[\pi^{\Spin(9)}_{x,y}]^{\Spin(7)'}\right)(4\Omega_{\Spin(9)} -
  3\Omega_{\Spin(8)}) &= (2x+y)^2 + 14(2x+y)\\
 &=\pi^{O(16)}_{2x+y}(\Omega_{O(16)}).
\end{aligned}\end{equation}
The last equality can be established exactly as in
\eqref{e:extraHspec}.
\end{subequations} 

Here is how the $\Spin(9)$ and $O(16)$ calculations fit together.
\begin{theorem}\label{thm:Ospin9cptbranch} Using the inclusion
  $\Spin(9) \subset O(16)$ given by the spin representation, we have
  $$\pi_a^{O(16)}|_{\Spin(9)} = \sum_{\substack{x\ge 0,\ y\ge 0
      \\[.2ex] 2x+y=a}} \pi_{x,y}^{\Spin(9)}.$$
  The contribution of these representations to the spectrum of the
  $O(16)$-invariant Laplacian $L_O$ is
  $$\begin{aligned}
    \pi^{O(16)}_a(\Omega_{O(16)}) &= a^2 + (16-2)a\\ &= (2x+y)^2 +
    14(2x+y) \\
    &=
    \left(\Spin(8)\cdot[\pi^{\Spin(9)}_{x,y}]^{\Spin(7)'}]\right)(4\Omega_{\Spin(9)}
      - 3\Omega_{\Spin(8)}).\end{aligned}$$
\end{theorem}

\section{The small $G_2$ calculation}
\label{sec:G2}
\setcounter{equation}{0}

\begin{subequations}\label{se:G2sphere}
Write $G_{2,c}$ for the $14$-dimensional compact connected Lie group
of type $G_2$. There is a $7$-dimensional real representation
$(\tau_{\mathbb R},W_{\mathbb R})$ of
$G_{2,c}$, whose (complexified) weights are zero and the six short
roots. The representation $\tau_{\mathbb R}$ preserves a positive
definite inner product, and so defines inclusions
\begin{equation}
  G_{2,c} \hookrightarrow SO(W_{\mathbb R}), \qquad  G_{2,c} \hookrightarrow
  \Spin(W_{\mathbb R}).
\end{equation}
The corresponding action of $G_{2,c}$ on $S^6$ is transitive. An
isotropy group is isomorphic to $SU(3)$; this is a subgroup generated
by a maximal torus and the long root $SU(2)$s. Therefore
\begin{equation}
  S^{6} = \{w\in W_{\mathbb R} \mid \langle v,v\rangle = 1\} \simeq G_{2,c}/SU(3).
\end{equation}
\end{subequations} 

\begin{subequations}\label{se:G2reps}
Here is the representation theory. Having identified a subgroup of
$G_{2,c}$ with $SU(3)$, we may as well take for our maximal torus in
$G_{2,c}$ the diagonal torus
\begin{equation}
  T = S(U(1)^3) \subset SU(3).
\end{equation}
The weights of $T$ are therefore
\begin{equation}
X^*(T) = \{\lambda=(\lambda_1,\lambda_2,\lambda_3) \mid \lambda_i -
\lambda_j \in {\mathbb Z},\quad  \lambda_1+\lambda_2 + \lambda_3 = 0\}.
\end{equation}
For each integer $a\ge 0$ there is an irreducible representation
\begin{equation}\label{eq:G2dim}
    \pi_{a} \text{\ highest wt\ }
    (2a/3,-a/3,-a/3), \quad
 \dim \pi_{a} =  \frac{(2a+5)\prod_{j=1}^4 (a+j)}{5!}
\end{equation}
Notice that the polynomial giving the dimension has degree $5$. In
fact it is exactly the polynomial of \eqref{eq:Rdim} giving the dimension of
$\pi_a^{O(7)}$.

Using the calculation of $\rho$ given in \eqref{eq:G2rho} below, we
find
\begin{equation}\label{eq:G2inflchar}
\hbox{infinitesimal character of\ }\pi_a =
((2a+5)/3,-(a+1)/3,-(a+4)/3).
\end{equation}

The key fact is that
\begin{equation}\label{e:G2key}
  \dim \pi_{a}^{SU(3)} = 1 \quad (a\ge 0), \qquad
  \dim\pi^{SU(3)} = 0 \quad (\pi \not\simeq \pi_{a}).
\end{equation}
Therefore
\begin{equation}
  L^2(S^{6}) \simeq \sum_{a\ge 0} \pi_{a}
\end{equation}
as representations of $G_{2,c}$.

\end{subequations}

\begin{subequations}\label{se:G2orbit}
Here is the orbit method perspective. Define
\begin{equation}
  a_{\text{orbit}} = a+5/2.
\end{equation}
Then it turns out that there is an element $\lambda(a_{\text{orbit}})
\in {\mathfrak g}_0^*$ (which we will
not attempt to write down) with the properties

\smallskip
\begin{equation} \begin{aligned}
\lambda(a_{\text{orbit}}) &\in ({\mathfrak  g}_0/{\mathfrak h}_0)^*\\
\lambda(a_{\text{orbit}}) &\ \text{is conjugate to}\\
&\quad a_{\text{orbit}}\cdot(2/3,-1/3,-1/3).
  \end{aligned}
\end{equation}

\smallskip

Consequently
\begin{equation}
  \pi_a = \pi(\text{orbit\ }
  \lambda(a_{\text{orbit}})).
\end{equation}
An aspect of the orbit method perspective is that the ``natural''
dominance condition is no longer $a \ge 0$ and but rather
\begin{equation}
  a_{\text{orbit}}  > 0 \iff a > -5/2.
\end{equation}
For the compact group $G_{2,c}$ we have
\begin{equation}
\pi(\text{orbit\ } \lambda(a_{\text{orbit}}))=0
\quad \text{if}\quad 0 > a > -5/2;
\end{equation}
so the difference is not important. But matters will be more
interesting in the noncompact case (Section \ref{sec:G2s}).
\end{subequations} 

\begin{subequations}\label{se:G2spec}
  Now we're ready for spectral theory.   We need to calculate
  $\pi_{a}(\Omega_{G_{2,c}})$. The sum of the positive roots is
  \begin{equation}\label{eq:G2rho}
    2\rho(G_{2,c}) = (10/3,-2/3,-8/3).
  \end{equation}
Because our highest weight is
  \begin{equation}
    \lambda =  (2a/3,-a/3,-a/3)
  \end{equation}
  we find
  \begin{equation}\label{eq:LG2}\begin{aligned}
    \pi_a(\Omega_{G_{2,c}}) &= 2a^2/3 + 10a/3 =
    2(a^2 +5a)/3\\
    &= (2/3)(a_{\text{orbit}}^2 - 25/4) \end{aligned}
  \end{equation}

\end{subequations} 

Here is how the $G_{2,c}$ and $O(7)$ calculations fit together.
\begin{theorem}
Using the inclusion $G_{2,c}\subset O(7)$, we have
  $$\pi^{O(7)}_a|_{G_{2,c}} = \pi^{G_{2,c}}_a.$$
The contribution of these representations to the spectrum of the
$O(7)$-invariant Laplacian $L_O$ is
$$\begin{aligned}\pi^{O(7)}_a(\Omega_{O(7)}) &= a^2 + 5a \\
  &= \pi^{G_{2,c}}_a(3\Omega_{G_{2,c}}/2).\end{aligned}$$
\end{theorem}
This is a consequence of the equality of dimensions observed at
\eqref{eq:G2dim}, together with the fact that the inclusion of
$G_{2,c}$ in $O(7)$ carries (some) short roots to (some) short roots.

\section{The big $G_2$ calculation}
\label{sec:bigG2}
\setcounter{equation}{0}

\begin{subequations} \label{se:bigG2sphere}
Suppose $n$ is an integer at least two. The group $\Spin(2n)$, or
equivalently the Lie algebra ${\mathfrak s}{\mathfrak p}{\mathfrak
  i}{\mathfrak n}(2n)$,  has an interesting
outer automorphism of order two: conjugation by the orthogonal matrix
\begin{equation}
\sigma = \Ad  \begin{pmatrix} 1 & 0 & \cdots & 0 & 0\\ 0 & 1 &
  \cdots & 0 & 0 \\ & &
  \ddots & & \\ 0 & 0 & \cdots & 1 & 0 \\ 0 & 0 & \cdots & 0 & -1
  \end{pmatrix}.
\end{equation}
The group of fixed points of $\sigma$ is the ``first $2n-1$ coordinates''
\begin{equation}
  \Spin(2n-1) = \Spin(2n)^{\sigma}.
\end{equation}
The automorphism $\sigma$ implements the automorphism of the Dynkin diagram
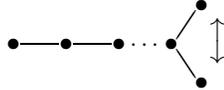
\setlength{\unitlength}{1cm}

\medskip
\begin{picture}(2.8,1)(-2.7,.2)
\multiput(-.1,.5)(.7,0){3}{\circle*{.15}}
\multiput(0,.5)(.7,0){2}{\line(1,0){.5}}
\put(1.45,.48){\small{\dots}}

\put(2.0,.5){\circle*{.15}}
\multiput(2.4,-.01)(0,1.03){2}{\circle*{.15}}
\put(2.07,.57){\line(2,3){.25}}
\put(2.07,.43){\line(2,-3){.25}}
\put(2.7,.22){\begin{rotate}{90}$\longleftrightarrow$\end{rotate}}

\end{picture}

\bigskip
\noindent exchanging the two short legs. If $n=4$, the Dynkin diagram

\medskip
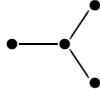
\setlength{\unitlength}{1cm}
\begin{picture}(2.8,1)(-2,.2)
\multiput(1.3,.5)(.5,0){1}{\circle*{.15}}
\put(1.4,.5){\line(1,0){.5}} 
\put(2.0,.5){\circle*{.15}}
\multiput(2.4,-.01)(0,1.03){2}{\circle*{.15}}
\put(2.07,.57){\line(2,3){.25}}
\put(2.07,.43){\line(2,-3){.25}}
\end{picture}

\bigskip
\noindent has two additional involutive automorphisms, exchanging the other two
pairs of legs. This gives rise to two additional (nonconjugate)
automorphisms $\sigma'$ and $\sigma''$ of $\Spin(8)$. Their fixed
point groups are isomorphic to $\Spin(7)$, but not conjugate to the
standard one (or to each other). We call them
\begin{equation}
  \Spin(8)^{\sigma'} = \Spin(7)', \qquad \Spin(8)^{\sigma''} = \Spin(7)''.
\end{equation}
The full automorphism group of the Dynkin diagram is the symmetric
group $S_3$; $\sigma_0$ and $\sigma_\pm$ are the three
transpositions, any two of which generate $S_3$. The fixed point group
of the full $S_3$ is
\begin{equation}
  \Spin(8)^{S_3} = G_{2,c} = \Spin(7) \cap \Spin(7)';
\end{equation}
this is a classical way to construct $G_{2,c}$. It follows that
\begin{equation}
S^7 = \Spin(8)/\Spin(7) \supset \Spin(7)'/G_{2,c}.
\end{equation}
Because the last homogeneous space is also seven-dimensional, the
inclusion is an equality
\begin{equation}
  S^7 = \Spin(7)'/G_{2.c}.
\end{equation}
\end{subequations} 

\begin{subequations}\label{se:bigG2reps}
Here is the representation theory. We take for our maximal torus in
$\Spin(7)'$ the double cover $T_+$ of
\begin{equation}
  SO(2)^3 \subset SO(7).
\end{equation}
The weights of $T_+$ are
\begin{equation}
X^*(T) = \{\lambda=(\lambda_1,\lambda_2,\lambda_3) \mid \lambda_i \in
{\mathbb Z} \text{\ (all $i$) or \  } \lambda_i \in
{\mathbb Z}+1/2 \text{\ (all $i$)}\}.
\end{equation}
For each integer $a\ge 0$ there is an irreducible representation
\begin{equation}\begin{aligned}
    \pi^{\Spin(7)'}_{a} &\text{\ highest wt\ }
    (a/2,a/2,a/2), \\
 \dim \pi^{\Spin(7)'}_{a} &=  \frac{(a+3)\prod_{j=1}^5(a+j)}{3\cdot 5!}
\end{aligned}\end{equation}
Notice that the polynomial giving the dimension has degree $6$; in
fact it is exactly the polynomial \eqref{eq:Rdim} giving the dimension of
$\pi_a^{O(8)}$.

Using the calculation of $\rho$ given in \eqref{eq:spinrho} below (or
in \eqref{eq:Orho})  we find
\begin{equation}\label{eq:bigG2inflchar}
\hbox{infl.~char.}(\pi^{\Spin(7)'}_a) =
((a+5)/2,(a+3)/2,(a+1)/2).
\end{equation}

The key fact is that
\begin{equation}\label{e:bigG2key}
  \pi_{a}^{O(8)}|_{\Spin(7)'} = \pi_a^{\Spin(7)'}.
\end{equation}
Therefore
\begin{equation}
  L^2(S^{7}) \simeq \sum_{a\ge 0} \pi^{\Spin(7)'}_{a}
\end{equation}
as representations of $\Spin(7)'$.

\end{subequations}

\begin{subequations}\label{se:bigG2orbit}
Here is the orbit method perspective. Define
\begin{equation}
  a_{\text{orbit}} = a+3.
\end{equation}
Then it turns out that there is a $7\times 7$ skew-symmetric real
matrix $\lambda(a_{\text{orbit}})$ (which we will
not attempt to write down) with the properties

\smallskip
\begin{equation} \begin{aligned}
\lambda(a_{\text{orbit}}) &\in ({\mathfrak g}_0/{\mathfrak
  h}_0)^*\\ \lambda(a_{\text{orbit}}) &\quad\text{has eigenvalues}\ \pm
a_{\text{orbit}}/4 \ \text{(three times)}.
  \end{aligned}
\end{equation}

\smallskip

Consequently
\begin{equation}
  \pi_{a}^{\Spin(7)'} = \pi(\text{orbit\ }
  \lambda(a_{\text{orbit}})).
\end{equation}
An aspect of the orbit method perspective is that the ``natural''
dominance condition is no longer $a \ge 0$ and but rather
\begin{equation}
  a_{\text{orbit}}  > 0 \iff a > -3.
\end{equation}
For the compact group $G_{2,c}$ we have
\begin{equation}
\pi(\text{orbit\ } \lambda(a_{\text{orbit}}))=0
\quad \text{if}\quad 0 > a > -3;
\end{equation}
so the difference is not important. But matters will be more
interesting in the noncompact case (Section \ref{sec:bigncG2}).
\end{subequations} 

\begin{subequations}\label{se:bigG2spec}
  Now we're ready for spectral theory.   We need to calculate
  $\pi_{a}(\Omega_{\Spin(7)'})$. The sum of the positive roots is
  \begin{equation}\label{eq:spinrho}
    2\rho(\Spin(7)') = (5,3,1).
  \end{equation}
Because our highest weight is
  \begin{equation}
    \lambda =  (a/2,a/2,a/2)
  \end{equation}
  we find
  \begin{equation}\label{eq:LbigG2}\begin{aligned}
    \pi^{\Spin(7)'}_a(\Omega_{\Spin(7)'}) &= 3a^2/4 + 9a/2 =
    3(a^2 +6a)/4\\ &= (3/4)(a_{\text{orbit}}^2 - 9)\\ &=
    (3/4)\pi^{O(8)}_a(\Omega_{O(8)}). 
\end{aligned}\end{equation}

\end{subequations} 

Here is a summary.
\begin{theorem}
Using the inclusion $\Spin(7)'\subset O(8)$, we have
  $$\pi^{O(8)}_a|_{\Spin(7)'} = \pi^{\Spin(7)'}_a.$$
The contribution of these representations to the spectrum of the
$O(8)$-invariant Laplacian $L_O$ is
$$\begin{aligned}\pi^{O(8)}_a(\Omega_{O(8)}) &= a^2 + 6a \\
  &= \pi^{\Spin(7)'}_a(4\Omega_{\Spin(7)'}/3).\end{aligned}$$
\end{theorem}

\section{Invariant differential operators}
\label{sec:invt}
\setcounter{equation}{0}

\begin{subequations}\label{se:invt}

  Suppose $H\subset G$ is a closed subgroup of a Lie group $G$. Write
  \begin{equation}
    {\mathbb D}(G/H) = \text{$G$-invariant differential operators on
      $G/H$},
  \end{equation}
an algebra. Following for example Helgason \cite{GGA}*{pages
  274--275}, we wish to understand this algebra and its spectral
theory as a way to understand functions on $G/H$. A first step is to
describe the algebra in terms of the Lie algebras of $G$ and $H$. This
is done in \cite{Hinvt} when $H$ is reductive in $G$ (precisely, when
the Lie algebra ${\mathfrak h}_0$ has an $\Ad(H)$-stable complement in
${\mathfrak g}_0$). Ways to remove this hypothesis have been understood
for a long time; we follow the nice account in \cite{invt}.

Write
  \begin{equation}\begin{aligned}
    {\mathfrak g}_0 &= \Lie(G) = \text{real left-invariant vector
      fields on $G$}\\
{\mathfrak g} &= {\mathfrak g}_0\otimes_{\mathbb R}{\mathbb C} =
\text{complex left-invariant vector fields on $G$}
  \end{aligned}\end{equation}
These vector fields act on functions by differentiating ``on the
right:''
\begin{equation}
  (Xf)(g) = \frac{d}{dt}\left( f(g\exp(tX))\right)|_{t=0} \qquad (X\in
  {\mathfrak g}_0).
\end{equation}
As usual we can therefore identify the enveloping algebra
\begin{equation}
U({\mathfrak g}) = \text{left-invariant complex differential operators
  on $G$}.
\end{equation}

We can identify
\begin{equation}
  C^\infty(G/H) = \{f\in C^\infty(G) \mid f(xh) = f(x)\quad (x\in G,
  h\in H)\}.
\end{equation}

Now consider the space
\begin{equation}
  I(G/H) =_{\text{def}} \left[ U({\mathfrak g})\otimes_{U({\mathfrak
        h})}{\mathbb C} \right]^{\Ad(H)\otimes 1}
\end{equation}
Before we pass to $\Ad(H)$-invariants, we have only a left
$U({\mathfrak g})$ module: no algebra structure. But
$\Ad(H)$-invariants inherit the algebra structure from $U({\mathfrak
  g})\otimes_{\mathbb C} {\mathbb C}$; so $I(G/H)$ is an algebra. The
natural action
\begin{equation}
  U({\mathfrak g})\otimes C^\infty(G) \rightarrow C^\infty(G)
\end{equation}
(which is a left algebra action, but comes by differentiating on the
right) restricts to a left algebra action
\begin{equation}
  I(G/H) \otimes C^\infty(G/H) \rightarrow C^\infty(G/H)
\end{equation}
on the subspace $C^\infty(G/H) \subset C^\infty(G)$.

Suppose more generally that $(\tau,V_\tau)$ is a finite-dimensional (and
therefore smooth) representation of $H$. Then
\begin{equation}
  {\mathcal V}_\tau = G \times_H V_\tau
\end{equation}
is a $G$-equivariant vector bundle on $G/H$. The space of smooth sections is
\begin{equation}
  C^\infty({\mathcal V}_\tau) = \{f\in C^\infty(G,V_\tau) \mid f(xh) =
  \tau(h)^{-1} f(x) \quad (x\in G, h\in H)\}.
\end{equation}

Now consider the space
\begin{equation}
  I^\tau(G/H) = \left[ U({\mathfrak g})\otimes_{U({\mathfrak h})}\End(V_\tau)
    \right]^{(\Ad\otimes \Ad)(H)}
\end{equation}
(The group $H$ acts by automorphisms on both the algebra $U({\mathfrak
  g})$ and the algebra $\End(V_\tau)$, in the latter case by
conjugation by the operators $\tau(h)$. The $H$-invariants are taken
for the tensor product of these two actions.) Before we pass to
$\Ad(H)$-invariants, we have only a left
$U({\mathfrak g})$ module: no algebra structure. But
$\Ad(H)$-invariants inherit the algebra structure from $U({\mathfrak
  g})\otimes_{\mathbb C} \End(V_\tau)$; so $I^\tau(G/H)$ is an algebra. The
natural action
\begin{equation}
  [U({\mathfrak g})\otimes_{\mathbb C} \End(V_\tau)] \otimes
  C^\infty(G,V_\tau) \rightarrow C^\infty(G,V_\tau)
\end{equation}
(which is a left algebra action, but comes by differentiating on the
right) restricts to a left algebra action
\begin{equation}\label{eq:bundleinvt}
  I^\tau(G/H)\otimes C^\infty({\mathcal V}_\tau) \rightarrow
  C^\infty({\mathcal V}_\tau).
\end{equation}
\end{subequations} 

\begin{proposition}[Helgason \cite{Hinvt}*{Theorem
    10}; \cite{HinvtBull}*{pages 758--759}; Koornwinder
    \cite{invt}*{Theorem 2.10}]\label{prop:invt} Suppose $H$ is a
  closed subgroup of the Lie group $G$. The action \eqref{eq:bundleinvt}
  identifies the algebra $I^\tau(G/H)$ with
  $${\mathbb D}^\tau(G/H) = \text{$G$-invariant differential operators
    on the vector bundle ${\mathcal V}_\tau$}.$$

  The action of $I^\tau(G/H)$ on formal power series sections of
  ${\mathcal V}_\tau$ at the identity is a faithful action.
\end{proposition}

\begin{subequations}\label{se:Helgharm}
Helgason's idea for invariant harmonic analysis (see for example
\cite{GGA}*{Introduction}) is to understand the
spectral theory of the algebra $I(G/H)={\mathbb D}(G/H)$ on
$C^\infty(G/H)$; or, more generally, of ${\mathbb D}^\tau(G/H)$ on
smooth sections of ${\mathcal V}_\tau$. Suppose for example that
${\mathbb D}(G/H)$ is {\em abelian}, and fix an algebra homomorphism
\begin{equation}
  \lambda \colon {\mathbb D}(G/H) \rightarrow {\mathbb C}, \qquad
  \lambda \in \Max\Spec({\mathbb D}(G/H)).
\end{equation}
Then the collection of simultaneous eigenfunctions
\begin{equation}
  C^\infty(G/H)_\lambda =_{\text{def}} \{f\in C^\infty(G/H) \mid Df =
  \lambda(D)f \mid D\in {\mathbb D}(G/H)\}
\end{equation}
is naturally a representation of $G$ (by left translation). The
question is for which $\lambda$ the space $C^\infty(G/H)_\lambda$ is
nonzero; and more precisely, what representation of $G$ it carries.
We can define
\begin{equation}
  \Spec(G/H) = \{\lambda \in \Max\Spec({\mathbb D}(G/H)) \mid
  C^\infty(G/H)_\lambda \ne 0\}.
\end{equation}
All of these remarks apply equally well to vector bundles.
\end{subequations}

\begin{subequations}\label{se:constructinvts1}
How can we identify interesting or computable invariant differential
operators? The easiest way is using the center of the enveloping
algebra
\begin{equation}\label{eq:Zg}
  {\mathfrak Z}({\mathfrak g}) =_{\text{def}} U({\mathfrak g})^G.
\end{equation}
(If $G$ is disconnected, this may be a proper subalgebra of the
center.) The obvious map
\begin{equation}\label{eq:Zgi}
  i_G\colon  {\mathfrak Z}({\mathfrak g}) \rightarrow I^\tau(G/H),
  \qquad z \mapsto z\otimes I_{V_\tau}
\end{equation}
is an algebra homomorphism. Here is how the spectral theory of the
differential operators $i_G( {\mathfrak Z}({\mathfrak g}))$ is related
to representation theory. Suppose that $(\pi,E_\pi)$ is a smooth
irreducible representation of $G$. Under a variety of mild assumptions
(for example, if $G$ is reductive and $\pi$ is quasisimple) there is a
homomorphism
\begin{equation}
  \chi_\pi\colon  {\mathfrak Z}({\mathfrak g}) \rightarrow {\mathbb C}
\end{equation}
called the {\em infinitesimal character of $\pi$} so that
\begin{equation}
  d\pi(z) = \chi_\pi(z)\cdot I_{E_\pi}.
\end{equation}
Suppose now that there is a $G$-equivariant inclusion
\begin{equation}\label{eq:harmanalysis}
  j_G\colon E_\pi \rightarrow C^\infty(G/H,{\mathcal V}_\tau).
\end{equation}
Finding inclusions like \eqref{eq:harmanalysis} is one of the things
harmonic analysis is about. One reason we care about it is the
consequences for spectral theory:
\begin{equation}\label{eq:spectral1}
  i_G(z) \text{\ acts on $j_G(E_\pi) \subset C^\infty({\mathcal V}_\tau)$ by
    the scalar $\chi_\pi(z)$} \qquad (z \in  {\mathfrak Z}({\mathfrak
    g})).
  \end{equation}
\end{subequations}

\begin{subequations}\label{se:constructinvts2}
Here is a generalization. Suppose $G_1$ is a subgroup of $G$
normalized by $H$:
\begin{equation}
  G_1 \subset G, \qquad \Ad(H)(G_1) \subset G_1.
\end{equation}
(The easiest way for this to happen is for $G_1$ to contain $H$.) Then
$H$ acts on ${\mathfrak Z}({\mathfrak g}_1)$, so we get
\begin{equation}\label{eq:Zg1i}
  i_{G_1}\colon  {\mathfrak Z}({\mathfrak g}_1)^H \rightarrow I^\tau(G/H),
  \qquad z_1 \mapsto z_1\otimes I_{V_\tau}.
\end{equation}
These invariant differential operators are acting along the
submanifolds
\begin{equation}
  xG_1/(G_1\cap H) \subset G/H \qquad (x\in G)
\end{equation}
of $G/H$. An example is the first coordinate $G_1=U(1)$ introduced in
\eqref{se:Ureps}, for $H=U(n-1)$. The operator $\Omega_{U(1)}$ on
$S^{2n-1}$ (acting along the fibers of the map $S^{2n-1} \rightarrow
{\mathbb C}{\mathbb P}^{n-1}$) is one of these new invariant
operators. A more interesting example is $G_1=Sp(1)\times Sp(1)$ studied
in \eqref{e:extraHreps}.

Here is how the spectral theory of these new operators is related to
representation theory. The map \eqref{eq:harmanalysis} is (by
Frobenius reciprocity) the same thing as an $H$-equivariant map
\begin{equation}
  j_H\colon E_\pi \rightarrow V_\tau
\end{equation}
or equivalently
\begin{equation}
  j_H^*\colon V_\tau^* \rightarrow E_\pi^*.
\end{equation}
It makes sense to define
\begin{equation}
  (E_\pi^*)^{G_1,j_H} = \text{$G_1$ representation generated by
    $j_H^*(V_\tau^*)$} \subset \pi^*.
\end{equation}
If the $G_1$ representation $(\pi^*)^{G_1,j_H}$ has infinitesimal character
$\chi_1^*$ (the contragredient of the infinitesimal character $\chi_1$),
then
\begin{equation} \label{eq:spectral2}
  i_{G_1}(z_1) \text{\ acts on $j_G(E_\pi) \subset C^\infty({\mathcal V}_\tau)$ by
    the scalar $\chi_1(z_1)$} \qquad (z_1 \in  {\mathfrak Z}({\mathfrak
    g}_1)^H).
\end{equation}

The homomorphisms $i_G$ of \eqref{eq:Zgi} and \eqref{eq:Zg1i} define an
algebra homomorphism from the abstract (commutative) tensor product
algebra
\begin{equation}\label{eq:Zgprodi}
i_G\otimes i_{G_1} \colon {\mathfrak Z}({\mathfrak g}) \otimes_{\mathbb
  C}{\mathfrak Z}({\mathfrak g}_1) \rightarrow I^\tau(G/H).
\end{equation}
The reason for this is that ${\mathfrak Z}({\mathfrak g})$ commutes
with all of $U({\mathfrak g})$.
\end{subequations}

Now that we understand the relationship between representations in
$C^\infty({\mathcal V}_\tau)$ and the spectrum of invariant
differential operators, let us see what the results of Sections
\ref{sec:R}--\ref{sec:bigG2} can tell us: in particular, about the
kernel of the homomorphism $i_G\otimes i_{G_1}$ of \eqref{eq:Zgprodi}.
\begin{subequations}\label{se:Rdiff}
We begin with $G=O(n)$, $H=O(n-1)$ as in Section
\ref{sec:R}. Write $n=2m+\epsilon$, with $\epsilon=0$ or
$1$. A maximal torus in $G$ is
\begin{equation}
  T=SO(2)^m,\qquad {\mathfrak t}_0 = {\mathbb
    R}^m, \qquad {\mathfrak t} = {\mathbb C}^m.
\end{equation}
The Weyl group $W(O(n))$ acts by permutation and sign changes on
these $m$ coordinates. Harish-Chandra's theorem identifies
\begin{equation}\label{eq:HCO}
  {\mathfrak Z}({\mathfrak g}) \simeq S({\mathfrak t})^{W(O(n))} =
  {\mathbb C}[x_1,\cdots,x_m]^{W(O(n))}.
\end{equation}
Therefore
\begin{equation}
  \hbox{(maximal ideals in ${\mathfrak Z}({\mathfrak g})$})
  \leftrightarrow {\mathbb C}^m/W(O(n)).
\end{equation}
Suppose $z\in {\mathfrak Z}({\mathfrak g})$ corresponds to $p\in
{\mathbb C}[x_1,\cdots,x_m]^{W(O(n))}$ by \eqref{eq:HCO}. According to
\eqref{eq:spectral1} and \eqref{eq:Oinfchar}, the invariant
differential operator $i_G(z)$ will act on $\pi_a\subset
C^\infty(G/H)$ by the scalar
$$  p(a+(n-2)/2, (n-4)/2,\cdots,(n-2m)/2).$$
Recalling that $n-2m=\epsilon=0$ or $1$, we write this as
\begin{equation}\label{eq:Rscalar}
  p(a+(n-2)/2, (n-4)/2,\cdots,\epsilon/2).
\end{equation}
\end{subequations}

Here is the consequence we want.
\begin{proposition}\label{prop:Rdiff} With notation as above,
  the polynomial
  $$p\in {\mathbb C}[x_1,\cdots,x_m]^{W(O(n))}$$
  vanishes on the (affine) line
  $$\{(\alpha,(n-4)/2,\cdots,\epsilon/2) \mid \alpha\in {\mathbb C}\}.$$
  if and only if $i_G(z)\in I(G/H)$ is equal to zero.
\end{proposition}
\begin{proof} The statement ``if'' is a consequence of
  \eqref{eq:Rscalar}: if the differential operator is zero, then $p$
  must vanish at all the points $(a+(n-2)/2, (n-4)/2,\cdots)$ with $a$
  a non-negative integer. These points are Zariski dense in the
  line. For ``only if,'' the vanishing of the polynomial makes the
  differential operator act by zero on all the subspaces $\pi_a\subset
  C^\infty(G/H)$. The sum of these subspaces is dense (for example as
  a consequence of \eqref{eq:Osphere}); so the differential operator
  acts by zero. The faithfulness statement in Proposition
  \ref{prop:invt} then implies that $i_G(z)=0$. \end{proof}
\begin{corollary}\label{cor:Rdiff}
\addtocounter{equation}{-1}
\begin{subequations}
  The $O(n)$ infinitesimal characters
  factoring to $i_G({\mathfrak Z}({\mathfrak g}))$ are indexed by
  weights
\begin{equation}\label{eq:Rinfchar}
  (\alpha,(n-4)/2,\cdots,\epsilon/2) \qquad (\alpha \in {\mathbb C}).
\end{equation}

  Suppose $(\pi,E_\pi)$ is a representation of ${\mathfrak
    o}(n,{\mathbb C})$ having an infinitesimal character, and that
  $(E_\pi^*)^{{\mathfrak o}(n-1,{\mathbb C})} \ne 0$. Then $\pi$ has
  infinitesimal character of the form \eqref{eq:Rinfchar}.
  \end{subequations}
\end{corollary}

Exactly the same arguments apply to the other examples treated in
Sections \ref{sec:R}--\ref{sec:bigG2}. We will just state the
conclusions.

\begin{subequations}\label{se:Cdiff}
Suppose $G=U(n)$, $H=U(n-1)$ as in Section
\ref{sec:C}. A maximal torus in $G$ is
\begin{equation}
  T=U(1)^n,\qquad {\mathfrak t}_0 = {\mathbb
    R}^n, \qquad {\mathfrak t} = {\mathbb C}^n.
\end{equation}
The Weyl group $W(U(n))$ acts by permutation on
these $n$ coordinates. Harish-Chandra's theorem identifies
\begin{equation}\label{eq:HCU}
  {\mathfrak Z}({\mathfrak g}) \simeq S({\mathfrak t})^{W(U(n))} =
  {\mathbb C}[x_1,\cdots,x_n]^{W(U(n))}.
\end{equation}
Therefore
\begin{equation}
  \hbox{(maximal ideals in ${\mathfrak Z}({\mathfrak g})$})
  \leftrightarrow {\mathbb C}^n/W(U(n)).
\end{equation}
Suppose $z\in {\mathfrak Z}({\mathfrak g})$ corresponds to $p\in
{\mathbb C}[x_1,\cdots,x_n]^{W(U(n))}$ by \eqref{eq:HCU}. According to
\eqref{eq:spectral1} and \eqref{eq:Uinflchar}, the invariant
differential operator $i_G(z)$ will act on $\pi_{b,c}\subset
C^\infty(G/H)$ by the scalar
\begin{equation}\label{eq:Cscalar}
  p((b+(n-1))/2, (n-3)/2,\cdots,-(n-3)/2,-(c+(n-1))/2).
\end{equation}
\end{subequations}

\begin{proposition}\label{prop:Cdiff} With notation as above,
  the polynomial
  $$p\in {\mathbb C}[x_1,\cdots,x_n]^{W(U(n))}$$
  vanishes on the (affine) plane
  $$\{(\xi,(n-3)/2,\cdots,-(n-3)/2,-\tau) \mid (\xi,\tau)\in {\mathbb C}^2\}.$$
  if and only if $i_G(z)\in I(G/H)$ is equal to zero.
\end{proposition}

\begin{corollary}\label{cor:Cdiff}
\addtocounter{equation}{-1}
\begin{subequations}
  The $U(n)$ infinitesimal characters
  factoring to $i_G({\mathfrak Z}({\mathfrak g}))$ are indexed by
  weights
  \begin{equation}\label{eq:Cinfchar}
    (\xi,(n-3)/2,\cdots,-(n-3)/2,-\tau) \qquad ((\xi,\tau) \in
    {\mathbb C}^2).
  \end{equation}

  Suppose $(\gamma,F_\gamma)$ is a representation of ${\mathfrak
    u}(n,{\mathbb C})$ having an infinitesimal character, and that
  $(F_\gamma^*)^{{\mathfrak u}(n-1,{\mathbb C})} \ne 0$. Then
  $F_\gamma$ has infinitesimal character of the form \eqref{eq:Cinfchar}. The
  parameters $\xi$ and $\tau$ may be determined as follows. The
  central character of $\gamma$ (scalars by which the one-dimensional center of
  the Lie algebra acts) is given by $\xi-\tau$. If in addition
  $F_\gamma \subset E_\pi$ for
  some representation $(\pi,E_\pi)$ of ${\mathfrak o}(2n,{\mathbb C})$
  as in Corollary \ref{cor:Rdiff}, then we may take $\xi+\tau =
  \alpha$. (Replacing $\alpha$ by the equivalent infinitesimal
  character parameter $-\alpha$
  has the effect of interchanging $\xi$ and $-\tau$, which defines an
  equivalent infinitesimal character parameter.)
  \end{subequations}
\end{corollary}

\begin{subequations}\label{se:Hdiff}
Suppose next that $G=Sp(n)\times Sp(1)$, $H=Sp(n-1)\times
Sp(1)_\Delta$ as in Section \ref{sec:H}. A maximal torus in $G$ is
\begin{equation}
  T=U(1)^n\times U(1),\qquad {\mathfrak t}_0 = {\mathbb
    R}^n\times {\mathbb R}, \qquad {\mathfrak t} = {\mathbb C}^n
  \times {\mathbb C}.
\end{equation}
The Weyl group $W(Sp(n)\times Sp(1))$ acts by sign changes on all
$n+1$ coordinates, and permutation of the first $n$
coordinates. Harish-Chandra's theorem identifies
\begin{equation}\label{eq:HCSp}
  {\mathfrak Z}({\mathfrak g}) \simeq S({\mathfrak t})^{W(Sp(n)\times Sp(1))} =
  {\mathbb C}[x_1,\cdots,x_n,y]^{W(Sp(n)\times Sp(1))}.
\end{equation}
Therefore
\begin{equation}
  \hbox{(maximal ideals in ${\mathfrak Z}({\mathfrak g})$})
  \leftrightarrow {\mathbb C}^{n+1}/W(Sp(n)\times Sp(1)).
\end{equation}
Suppose $z\in {\mathfrak Z}({\mathfrak g})$ corresponds to $p\in
{\mathbb C}[x_1,\cdots,x_n,y]^{W(Sp(n)\times Sp(1))}$ by
\eqref{eq:HCSp}. According to \eqref{eq:spectral1} and
\eqref{eq:Hinflchar}, the invariant differential operator $i_G(z)$
will act on $\pi_{d,e}\subset C^\infty(G/H)$ by the scalar
\begin{equation}\label{eq:Hscalar}
  p((d+n, e+(n-1),n-2,\cdots,1),(d-e+1)).
\end{equation}
\end{subequations}

\begin{proposition}\label{prop:Hdiff} With notation as above,
  the polynomial
  $$p\in {\mathbb C}[x_1,\cdots,x_n,y]^{W(Sp(n)\times Sp(1))}$$
  vanishes on the (affine) plane
  $$\{(\xi,\tau,n-2,\cdots,1)(\xi-\tau) \mid (\xi,\tau)\in {\mathbb C}^2\}.$$
  if and only if $i_G(z)\in I(G/H)$ is equal to zero.
\end{proposition}

\begin{corollary}\label{cor:Hdiff}
\addtocounter{equation}{-1}
\begin{subequations}
  The infinitesimal characters for
  $Sp(n)\times Sp(1)$ which factor to $i_G({\mathfrak Z}({\mathfrak
    g}))$ are indexed by weights
  \begin{equation}\label{eq:Hinfchar}
    (\xi,\tau,n-2,\cdots,1)(\xi-\tau) \qquad ((\xi,\tau) \in {\mathbb
      C}^2).
  \end{equation}

   Suppose $(\gamma,F_\gamma)$ is a representation of ${\mathfrak
     s}{\mathfrak p}(n,{\mathbb C}) \times {\mathfrak
     s}{\mathfrak p}(1,{\mathbb C})$ having an infinitesimal
   character, and that $(F_\gamma^*)^{{\mathfrak s}{\mathfrak
       p}(n-1,{\mathbb C}) \times {\mathfrak
     s}{\mathfrak p}(1,{\mathbb C})_\Delta} \ne 0$. Then $F_\gamma$
   has infinitesimal character of the form \eqref{eq:Hinfchar};
   $\xi-\tau$ is the infinitesimal character of the ${\mathfrak
     s}{\mathfrak p}(1,{\mathbb C})$ factor. If in
   addition $F_\gamma \subset 
   E_\pi$ for some representation $(\pi,E_\pi)$ of ${\mathfrak
     o}(4n,{\mathbb C})$ as in Corollary \ref{cor:Rdiff}, then we may take
$\xi+\tau = \alpha$.
  \end{subequations}
\end{corollary}

This is a good setting in which to consider the more general invariant
differential operators from \eqref{se:constructinvts2}.
\begin{subequations}\label{se:HG1diff}
Suppose in that general setting that $G_1$ is reductive, and choose a Cartan
subalgebra ${\mathfrak t}_1\subset {\mathfrak g}_1$, with (finite) Weyl group
\begin{equation}
  W(G_1) =_{\text{def}} N_{G_1({\mathbb C})}({\mathfrak
    t}_1)/Z_{G_1({\mathbb C})}({\mathfrak t_1}) \subset
  \Aut({\mathfrak t}_1), \qquad {\mathfrak Z}({\mathfrak g}_1)
  \simeq S({\mathfrak t})^{W_1}.
\end{equation}
The adjoint action of $H$ on $G_1$ defines another Weyl group,
which normalizes $W(G_1)$:
\begin{equation}
  W(G_1) \triangleleft W_H(G_1) =_{\text{def}} N_{H({\mathbb C})}({\mathfrak
    t}_1)/Z_{H({\mathbb C})}({\mathfrak t_1}) \subset \Aut({\mathfrak
    t}_1), \qquad {\mathfrak Z}({\mathfrak g}_1)^H
  \simeq S({\mathfrak t}_1)^{W_H(G_1)}.
\end{equation}
Under mild hypotheses (for example $G_1$ is reductive algebraic and the
adjoint action of $H$ is algebraic) then $W_H(G_1)$ is finite, so the
algebra ${\mathfrak Z}({\mathfrak g}_1)$ is finite over ${\mathfrak
  Z}({\mathfrak g}_1)^H$, and the maximal ideals in this smaller algebra
are given by evaluation at
\begin{equation}
  \mu\in {\mathfrak t}_1^*/W_H(G_1).
  \end{equation}
In the case $G_1=Sp(1)\times Sp(1)$, the adjoint action of $H$ on $G_1$ is
contained in that of $G_1$, so $W(G_1)=W_H(G_1)$, and ${\mathfrak Z}({\mathfrak
  g}_1)^H = {\mathfrak Z}({\mathfrak g}_1)$. We have
\begin{equation}
  T_1=U(1)^2, \qquad {\mathfrak t}_{1,0} = {\mathbb R}^2, \qquad {\mathfrak
    t}_1 = {\mathbb C}^2.
\end{equation}
The Weyl group $W(G_1)=W_H(G_1)$ acts by sign changes on each
coordinate, so the Harish-Chandra isomorphism is
\begin{equation}\label{eq:HCH}
  {\mathfrak Z}({\mathfrak g}_1)^H = {\mathfrak Z}({\mathfrak g}_1) \simeq
  S({\mathfrak t}_1)^{W(G_1)} = {\mathbb C}[u_1,u_2]^{W(G_1)}
\end{equation}

Suppose therefore that $z_1\in {\mathfrak Z}({\mathfrak g}_1)$
corresponds to $p_1\in {\mathbb C}[x_1,x_2]^{W(G_1)}$. According to
\eqref{eq:spectral2} and
\eqref{eq:Hsubinflchar}, the invariant differential operator $i_{G_1}(z_1)$
acts on $\pi^{Sp(n)\times Sp(1)}_{d,e} \subset C^\infty(G/H)$ by the
scalar
\begin{equation}\label{eq:HG1scalar}
  p_1(d-e+1,d-e+1).
\end{equation}
\end{subequations}

\begin{proposition}\label{prop:HG1diff} With notation as above,
  suppose that
  $$P\in {\mathbb C}[x_1,\cdots,x_n,y,u_1,u_2]^{W(G)\times W(G_1)},$$
  and write $Z\in {\mathfrak Z}({\mathfrak g})\otimes {\mathfrak
    Z}({\mathfrak g}_1)^H$ for the corresponding central element. Then $P$
  vanishes on the affine plane
  $$\{(\xi,\tau,n-2,\cdots,1)(\xi-\tau)(\xi-\tau,\xi-\tau)\mid
  (\xi,\tau)\in {\mathbb C}^2\}.$$
  if and only if $(i_G\otimes i_{G_1})(Z)\in I(G/H)$ is equal to zero.
\end{proposition}

\begin{corollary}\label{cor:HG1diff}
\addtocounter{equation}{-1}
\begin{subequations}
  In the setting $G/H =
  (Sp(n)\times Sp(1))/(Sp(n-1)\times Sp(1)_\Delta)$, $G_1=Sp(1)\times Sp(1)$,
  the characters of the tensor product algebra \eqref{eq:Zgprodi}
  which factor to the image in $I(G/H)$ are indexed by weights
  \begin{equation}\label{eq:HG1infchar}
    (\xi,\tau,n-2,\cdots,1)(\xi-\tau)(\xi-\tau,\xi-\tau).
  \end{equation}
 Here the first $n$ coordinates are giving the infinitesimal
  character for $Sp(n)$; the next is the infinitesimal character for
  the $Sp(1)$ factor of $G$; and the last two are the infinitesimal
  character for $G_1$.

  Suppose $(\gamma,F_\gamma)$ is an ${\mathfrak
    s}{\mathfrak p}(n,{\mathbb C})$ representation as in Corollary
  \ref{cor:Hdiff}. Then the ${\mathfrak g}_1$ representation
  generated by $(F_\gamma^*)^{{\mathfrak s}{\mathfrak p}(n-1,{\mathbb
      C})}$ has infinitesimal character $(\xi-\tau,\xi-\tau)$.
\end{subequations}
\end{corollary}

\begin{subequations}\label{se:octdiff}
Suppose next that $G=\Spin(9)$, $H=\Spin(7)'$ as in Section
\ref{sec:O}. A maximal torus in $G$ is
\begin{equation}
  T= \hbox{double cover of\ }SO(2)^4,\qquad {\mathfrak t}_0 = {\mathbb
    R}^4, \qquad {\mathfrak t} = {\mathbb C}^4.
\end{equation}
The Weyl group $W(\Spin(9))$ acts by permutation and sign changes on these
four coordinates. Harish-Chandra's theorem identifies
\begin{equation}\label{eq:HCoct}
  {\mathfrak Z}({\mathfrak g}) \simeq S({\mathfrak t})^{W(\Spin(9))} =
  {\mathbb C}[x_1,\cdots,x_4]^{W(\Spin(9))}.
\end{equation}
Therefore
\begin{equation}
  \hbox{(maximal ideals in ${\mathfrak Z}({\mathfrak g})$})
  \longleftrightarrow {\mathbb C}^{4}/W(\Spin(9)).
\end{equation}
Suppose $z\in {\mathfrak Z}({\mathfrak g})$ corresponds to $p\in
{\mathbb C}[x_1,\cdots,x_4]^{W(\Spin(9))}$ by
\eqref{eq:HCoct}. According to \eqref{eq:spectral1} and
\eqref{eq:octinflchar}, the invariant differential operator $i_G(z)$
will act on $\pi^{\Spin(9)}_{x,y}\subset C^\infty(G/H)$ by the scalar
\begin{equation}\label{eq:octscalar}
  p((2x+y+7)/2,(y+5)/2,(y+3)/2,(y+1)/2).
\end{equation}
\end{subequations}

\begin{proposition}\label{prop:octdiff} With notation as above,
  the polynomial
  $$p\in {\mathbb C}[x_1,\cdots,x_4]^{W(\Spin(9))}$$
  vanishes on the (affine) plane
  $$\{(\xi,\tau+5/2,\tau+3/2,\tau+1/2) \mid (\xi,\tau)\in
  {\mathbb C}^2\}.$$
  if and only if $i_G(z)\in I(G/H)$ is equal to zero.
\end{proposition}

\begin{corollary}\label{cor:octdiff}
\addtocounter{equation}{-1}
\begin{subequations}
  The infinitesimal characters for
  $\Spin(9)$ factoring to $i_G({\mathfrak Z}({\mathfrak
    g}))$ are indexed by weights
  \begin{equation}\label{eq:octinfchar}
    (\xi,\tau+5/2,\tau+3/2,\tau+1/2) \qquad ((\xi,\tau) \in {\mathbb C}^2).
  \end{equation}

     Suppose $(\gamma,F_\gamma)$ is a representation of ${\mathfrak
    s}{\mathfrak p}{\mathfrak i}{\mathfrak n}(9,{\mathbb C})$ having
     an infinitesimal character, and that
  $(F_\gamma^*)^{{\mathfrak h}({\mathbb C})} \ne 0$. Then $F_\gamma$ has
  infinitesimal character of the form \eqref{eq:Hinfchar}. If the
  ${\mathfrak s} {\mathfrak p} {\mathfrak i} {\mathfrak n} (8,{\mathbb
    C})$-module generated by $(F_\gamma^*)^{{\mathfrak h}({\mathbb
      C})}$ has a submodule with an infinitesimal character, then we
  may choose $\tau$ so that this infinitesimal character is
  \begin{equation}
    (\tau+3,\tau+2,\tau+1,\tau).
  \end{equation}
If in addition $F_\gamma \subset E_\pi$ for
  some representation $(\pi,E_\pi)$ of ${\mathfrak o}(16,{\mathbb C})$
  as in Corollary \ref{cor:Rdiff} (with infinitesimal character
  parameter $\alpha$) then we may choose $\xi=\alpha/2$.
  \end{subequations}
\end{corollary}

For the last two cases we write even less.
\begin{corollary}\label{cor:G2diff}
\addtocounter{equation}{-1}
\begin{subequations}
  When $G/H = G_{2,c}/SU(3)$, the
  infinitesimal characters for
  $G_{2,c}$ which factor to $i_G({\mathfrak Z}({\mathfrak
    g}))$ are indexed by weights
  \begin{equation}\label{eq:G2infchar}
    (2\xi,(1/2)-\xi,-(1/2)-\xi) \qquad (\xi \in {\mathbb C}).
  \end{equation}

 Suppose $(\gamma,F_\gamma)$ is a representation of ${\mathfrak
    g}_2({\mathbb C})$ having an infinitesimal character, and that
  $(F_\gamma^*)^{{\mathfrak u}(3,{\mathbb C})} \ne 0$. Then the
  infinitesimal character of $F_\gamma$ is of the form in
  \eqref{eq:G2infchar}. If in addition $F_\gamma \subset E_\pi$ for
  some representation $(\pi,E_\pi)$ of ${\mathfrak o}(7,{\mathbb C})$
  as in Corollary \ref{cor:Rdiff}, then we may take $\xi = \alpha/3$.
\end{subequations}
\end{corollary}

\begin{corollary}\label{cor:bigG2diff}
  \addtocounter{equation}{-1}
  \begin{subequations}
    When $G/H = \Spin(7)'/G_{2,c}$, the
  infinitesimal characters for
  $\Spin(7)'$ which factor to $i_G({\mathfrak Z}({\mathfrak
    g}))$ are indexed by weights
  \begin{equation}\label{eq:bigG2infchar}
    (\xi+1,\xi,\xi-1) \qquad (\xi \in {\mathbb C}).
  \end{equation}

 Suppose $(\gamma,F_\gamma)$ is a representation of ${\mathfrak
    s}{\mathfrak p}{\mathfrak i}{\mathfrak n}(7,{\mathbb C})'$ having
 an infinitesimal character, and that
  $(F_\gamma^*)^{{\mathfrak g}_2({\mathbb C})} \ne 0$. Then the
  infinitesimal character of $F_\gamma$ is of the form in
  \eqref{eq:bigG2infchar}. If in addition $F_\gamma \subset E_\pi$ for
  some representation $(\pi,E_\pi)$ of ${\mathfrak o}(8,{\mathbb C})$
  as in Corollary \ref{cor:Rdiff}, then we may take $\xi= \alpha/2$.
  \end{subequations}
\end{corollary}

\section{Changing real forms}
\label{sec:changereal}
\setcounter{equation}{0}
Results like \eqref{eq:spectral1} and its generalization
\eqref{eq:spectral2} explain why it is interesting to study the
representations of $G$ appearing in $C^\infty({\mathcal V}_\tau)$ and
the invariant differential operators on this space. In this section we
state our first method for doing that.

\begin{definition}\label{def:realform}
\addtocounter{equation}{-1}
\begin{subequations}\label{se:realform}
  Suppose $G_1$ and $G_2$ are Lie
  groups with closed subgroups $H_1$ and $H_2$. Assume that there is
  an isomorphism of complexified Lie algebras
\begin{equation}i\colon {\mathfrak g}_1 \buildrel{\sim}\over{\longrightarrow}
  {\mathfrak g}_2, \qquad i({\mathfrak h}_1) = {\mathfrak h}_2.
  \end{equation}
Finally, assume that $i$ identifies the Zariski closure of $\Ad(H_1)$
in $\Aut({\mathfrak g}_1)$ with the Zariski closure of $\Ad(H_2)$ in
$\Aut({\mathfrak g}_2)$. (This is automatic if $H_1$ and $H_2$ are
connected.) Then we say that the homogeneous space $G_2/H_2$ is {\em
  another real form} of the homogeneous space $G_1/H_1$.

Given representations $(\tau_i,V_{\tau_i})$ of $H_i$, we say that
${\mathcal V}_{\tau_2}$ is {\em another real form} of ${\mathcal
  V}_{\tau_1}$ if there is an isomorphism
\begin{equation}i\colon V_{\tau_1} \buildrel{\sim}\over{\longrightarrow}
V_{\tau_2}\end{equation}
respecting the actions of ${\mathfrak h}$, and identifying the
Zariski closure of $\Ad(H_1)$ in $\End(V_{\tau_1})$ with the Zariski
closure of $\Ad(H_2)$ in $\End(V_{\tau_2})$.

Whenever ${\mathcal V}_{\tau_2}$ is another real form of ${\mathcal
  V}_{\tau_1}$, we get an algebra isomorphism
\begin{equation}\label{eq:realformisom}
i\colon {\mathbb D}^{\tau_1}(G_1/H_1)
\buildrel{\sim}\over{\longrightarrow} {\mathbb D}^{\tau_2}(G_2/H_2).
\end{equation}
\end{subequations}
\end{definition}
We will use these isomorphisms together with results like Corollaries
\ref{cor:Cdiff}--\ref{cor:bigG2diff} (proven using compact homogeneous
spaces $G_1/H_1$) to control the possible representations appearing in
some noncompact homogeneous spaces $G_2/H_2$.

\section{Changing the size of the group}
\label{sec:size}
\setcounter{equation}{0}

Our second way to study representations and invariant differential
operators is this.
\begin{subequations}\label{se:size}
In the setting \eqref{se:invt}, suppose that $S\subset G$ is a closed
subgroup, and that
\begin{equation}\label{eq:bigsub1}
  \dim G/H = \dim S/(S\cap H).
\end{equation}
Equivalent requirements are
\begin{equation}\label{eq:bigsub2}
  {\mathfrak s}/({\mathfrak s}\cap {\mathfrak h}) = {\mathfrak
    g}/{\mathfrak h}
\end{equation}
or
\begin{equation}\label{eq:bigsub3}
  {\mathfrak s} + {\mathfrak h} = {\mathfrak g}
\end{equation}
or
\begin{equation}\label{eq:bigsub4}
  \hbox{$S/(S\cap H)$ is open in $G/H$}.
\end{equation}
Because of this open embedding, differential operators on $S/(S\cap
H)$ are more or less the same thing as differential operators on
$G/H$. The condition of $S$-invariance is weaker than the condition of
$G$-invariance, so we get natural inclusions
\begin{equation}\label{eq:opincl}
  {\mathbb D}(G/H) \hookrightarrow {\mathbb D}(S/(S\cap H)), \qquad {\mathbb
    D}^\tau(G/H) \hookrightarrow {\mathbb D}^\tau(S/(S\cap H)).
  \end{equation}
(notation as in \eqref{se:invt}). In terms of the
algebraic description of these operators given in Proposition
\ref{prop:invt}, notice first that the condition in \eqref{eq:bigsub2}
shows that the inclusion ${\mathfrak s} \hookrightarrow {\mathfrak g}$
defines an isomorphism
\begin{equation}
  U({\mathfrak s})\otimes_{{\mathfrak s}\cap{\mathfrak h}}\End(V_\tau)
  \simeq U({\mathfrak g})\otimes_{{\mathfrak h}}\End(V_\tau)
\end{equation}
Therefore
\begin{equation}\begin{aligned}
 \   [U({\mathfrak g})\otimes_{{\mathfrak
          h}}\End(V_\tau)]^{(\Ad\otimes \Ad)(H)} &\hookrightarrow
    [U({\mathfrak g})\otimes_{{\mathfrak
          h}}\End(V_\tau)]^{(\Ad\otimes\Ad)(S\cap H)} \\
  &\simeq [U({\mathfrak s})\otimes_{{\mathfrak s}\cap{\mathfrak
        h}}\End(V_\tau)]^{(\Ad\otimes\Ad)(S\cap H)}.
  \end{aligned} \end{equation}
That is,
\begin{equation}\label{eq:algincl}
  I^\tau(G/H) \hookrightarrow I^\tau(S/(S\cap H)).
\end{equation}
This algebra inclusion corresponds to the differential operator
inclusion \eqref{eq:opincl} under the identification of Proposition
\ref{prop:invt}.
\end{subequations} 

Here is a useful fact.

\begin{proposition} Let $G$ be a connected reductive Lie group,
and let $H$ and $S$ be closed connected reductive subgroups.
Assume the equivalent conditions (\ref{eq:bigsub1})-(\ref{eq:bigsub4}). Then
\begin{enumerate}
  \item $G=SH$, and
\item there is a Cartan involution for $G$ preserving both $S$ and
  $H$.
\end{enumerate}
\end{proposition}

\begin{proof} Part (1) is due to Onishchik \cite{Oni69}*{Theorem 3.1}.
For (2), since $H$ is reductive in $G$, there is a Cartan
involution $\theta_H$ for $G$ preserving $H$, and likewise there is
one $\theta_S$ preserving $S$. By the uniqueness of Cartan involutions
for $G$, $\theta_S$ is the conjugate of $\theta_H$ by some element $g\in G$,
which by (1) can be decomposed as $g=sh$. The $h$-conjugate of $\theta_H$,
which is also the $s^{-1}$-conjugate of $\theta_S$, has the required property.
\end{proof}

It follows from (1) that if $(G_c,S_c,H_c)$ is a triple of a compact Lie group
and two closed subgroups such that $G_c=S_cH_c$, and if $(G,S,H)$ is a triple
of real forms (that is, $G/S$ is a real form of $G_c/S_c$ and
$G/H$ a real form of $G_c/H_c$), then $S$ acts transitively on $G/H$.
Conversely, by (2) every transitive action on a reductive homogeneous space
$G/H$ by a reductive subgroup $S\subset G$ is obtained in this way.

In the following sections we shall apply this principle to the real hyperboloid
(\ref{eq:Hpq}), which is a real form of $S^{p+q-1}=O(p+q)/O(p+q-1)$.

The hypothesis that both $S$ and $H$ be reductive is certainly
necessary. Suppose for example that $S$ is a noncompact
real form of the complex reductive group $G$, and that
$H$ is a parabolic subgroup of $G$ (so that $S$ and $G$ are reductive,
but $H$ is not). Then $S$ has finitely many orbits on $G/H$
(\cite{Wolfflag}), and in particular has open orbits (so that the
conditions \eqref{eq:bigsub1}--\eqref{eq:bigsub4} are satisfied); but
the number of orbits is almost always greater than one (so $G \ne SH$).

\section{Classical hyperboloids}
\label{sec:Opq}
\setcounter{equation}{0}

\begin{subequations}\label{se:Opq}
In this section we recall the classical representation-theoretic
decomposition of functions on real hyperboloids: that is, on other
real forms of spheres. The spaces are
\begin{equation}\label{eq:Hpq}
  H_{p,q} = \{v\in {\mathbb R}^{p,q} \mid \langle
    v,v\rangle_{p,q} = 1\} = O(p,q)/O(p-1,q).
\end{equation}
Here $\langle,\rangle_{p,q}$ is the standard quadratic form of
signature $(p,q)$ on ${\mathbb R}^{p+q}$. The inclusion of the right
side of the equality in the middle is just given by the action of the
orthogonal group on the basis vector $e_1$; surjectivity is Witt's
theorem.  This realization of the hyperboloid is a symmetric space, so
the Plancherel decomposition is completely known. In particular, the
discrete series may be described as follows. To avoid degenerate
cases, we assume that
\begin{equation}\label{eq:Obigp}
  p \ge 2.
\end{equation}
There is a ``compact Cartan subspace'' with Lie algebra
\begin{equation}\label{eq:cptCartanOpq}
  {\mathfrak a}_c = \langle e_{12} - e_{21} \rangle.
\end{equation}
The first requirement is that
\begin{equation}
  {\mathfrak a}_c \subset {\mathfrak k} = {\mathfrak o}(p) \times
  {\mathfrak o}(q).
\end{equation}
That this is satisfied is a consequence of \eqref{eq:Obigp}. The
second requirement is that ${\mathfrak a}_c$ belongs to the $-1$
eigenspace of the involutive automorphism
\begin{equation}
  \sigma = \Ad\left(\diag(-1,1,1,\dots,1) \right)
\end{equation}
with fixed points the isotropy subgroup $O(p-1,q)$. (More precisely,
the group of fixed points of $\sigma$ is $O(1) \times O(p-1,q)$; so
our hyperboloid is a $2$-to-$1$ cover of the algebraic symmetric space
$O(p,q)/[O(1)\times O(p-1,q)]$. But the references also treat analysis on
this cover.)

For completeness, we mention that whenever
\begin{equation}\label{eq:Obigq}
  q \ge 1.
\end{equation}
there is another conjugacy class of Cartan subspace, represented by
\begin{equation}\label{eq:splCartanOpq}
  {\mathfrak a}_s = \langle e_{1,p+1} + e_{p+1,1} \rangle.
\end{equation}
This one is split, and corresponds to the continuous part of the
Plancherel formula.

The discrete series for the symmetric space $H_{p,q}$ is constructed
as follows. Using the compact Cartan subspace ${\mathfrak a}_c$,
construct a $\theta$-stable parabolic
\begin{equation}\label{eq:qOpq}
  {\mathfrak q}^{O(p,q)} = {\mathfrak l}^{O(p,q)} + {\mathfrak u}^{O(p,q)} \subset
  {\mathfrak o}(p+q,{\mathbb C});
\end{equation}
the corresponding Levi subgroup is
\begin{equation}
  L^{O(p,q)} = [O(p,q)]^{{\mathfrak a}_c} = SO(2) \times O(p-2,q)
\end{equation}
We will need notation for the characters of $SO(2)$:
\begin{equation}
  \widehat{SO(2)} = \{\chi_\ell \mid \ell \in {\mathbb Z}\}.
\end{equation}

The discrete series consists of certain irreducible representations
\begin{equation}
  A_{{\mathfrak q}^{O(p,q)}}(\lambda), \qquad \lambda\colon L^{O(p,q)} \rightarrow
  {\mathbb C}^\times.
\end{equation}
The allowed $\lambda$ are (first) those trivial on
\begin{equation}
  L^{O(p,q)} \cap O(p-1,q) = O(p-2,q).
\end{equation}
These are precisely the characters of $SO(2)$, and so are indexed by
integers $\ell \in {\mathbb Z}$. Second, there is a positivity
requirement
\begin{equation}
  \ell + (p+q-2)/2 > 0.
\end{equation}
We write
\begin{equation}\label{eq:OAq}\begin{aligned}
\lambda(\ell) &=_{\text{def}} \chi_\ell \otimes 1 \colon L^{O(p,q)}\rightarrow
       {\mathbb C}^\times, \\ \pi^{O(p,q)}_\ell &= A_{{\mathfrak
           q}^{O(p,q)}}(\lambda_\ell) \qquad 
  (\ell > (2-p-q)/2).
\end{aligned}\end{equation}
The infinitesimal character of this representation is
\begin{equation}
  \text{infl char}(\pi_\ell^{O(p,q)}) = (\ell +
  (p+q-2)/2, (p+q-4)/2, (p+q-6)/2,\cdots).
  \end{equation}
The discrete part of the Plancherel decomposition is
\begin{equation}\label{eq:Opqdisc}
  L^2(H_{p,q})_{\text{disc}} = \sum_{\ell > -(p+q-2)/2}
    \pi_\ell^{O(p,q)}.
\end{equation}
This decomposition appears in \cite{StrH}*{page 360}, and
\cite{RossH}*{page 449, Theorem 10, and page 471}. What Strichartz
calls $N$ and $n$ are for us $p$ and $q$; his $d$ is our $\ell$. What
Rossmann calls $q$ and $p$ are for us $p$ and $q$; his $\nu -\rho$ is
our $\ell$; and his $\rho$ is $(p+q-2)/2$. The
identification of the representations as cohomologically induced
may be found in \cite{Virr}*{Theorem 2.9}.
\end{subequations} 

\begin{subequations} \label{se:Opqorbit}
Here is the orbit method perspective. Just as for $O(n)$, we use a
trace form to identify ${\mathfrak g}_0^*$ with ${\mathfrak
  g}_0$. We find
\begin{equation}
  ({\mathfrak g}_0/{\mathfrak h}_0)^* \simeq {\mathbb R}^{p-1,q},
\end{equation}
respecting the action of $H=O(p-1,q)$. The orbits of $H$ of largest
dimension are given by the value of the quadratic form: positive for
the orbits represented by nonzero elements $x(e_{12}-e_{21})$ of the
compact Cartan subspace of \eqref{eq:cptCartanOpq}; negative for
nonzero elements of the split Cartan subspace
$y(e_{1,p+1}+ e_{p+1,1})$; and zero for the nilpotent element
$(e_{12}-e_{21}+e_{1,p+1}+e_{p+1,1})$.

Define
\begin{equation}\begin{aligned}
    \ell_{\text{orbit}} &= \ell + (n-2)/2,\\
    \lambda(\ell_{\text{orbit}}) &=
    \ell_{\text{orbit}}\cdot((e_{12}-e_{21})/2).
  \end{aligned}
\end{equation}
Then the coadjoint orbits for discrete series have representatives
in the compact Cartan subspace
\begin{equation}
  \pi_\ell^{O(p,q)} = \pi(\text{orbit}\ \lambda(\ell_{\text{orbit}})).
\end{equation}
Now this representation is an irreducible unitary cohomologically
induced representation whenever
\begin{equation}
  \ell_{\text{orbit}} > 0 \iff \ell > -(n-2)/2.
\end{equation}
One of the advantages of the orbit method picture is that the
condition $\ell_{\text{orbit}} >0$ is simpler than the one
$\ell>(-(n-2)/2)$ arising from more straightforward representation
theory as in \eqref{eq:OAq}. Of course we always need also the
integrality condition
\begin{equation}
  \ell_{\text{orbit}} \equiv (n-2)/2 \pmod{\mathbb Z} \iff \ell \equiv 0
  \pmod{\mathbb Z}.
  \end{equation}
\end{subequations} 

\begin{subequations} \label{se:Opqcont}
For completeness we mention also the continuous part of the Plancherel
decomposition. The split Cartan subspace ${\mathfrak a}_s$ (defined
above as long as $p$ and $q$ are each at least $1$) gives rise to a
real parabolic subgroup
\begin{equation}\label{eq:POpq}
  P^{O(p,q)} = M^{O(p,q)} A_s N^{O(p,q)}, \qquad M^{O(p,q)} = \{\pm
  1\} \times O(p-1,q-1).
\end{equation}
Here $A_s = \exp({\mathfrak a}_s) \simeq {\mathbb R}$, and $\{\pm 1\}$
is
$$O(1)_\Delta \subset O(1)\times O(1) \subset O(1,1);$$
we have
$$\{\pm 1\} \times A_s = SO(1,1) \simeq {\mathbb R}^\times,$$
an algebraic split torus. Therefore
\begin{equation}\label{eq:POpqB}
  P^{O(p,q)} = SO(1,1)\times O(p-1,q-1)\times N^{O(p,q)}.
\end{equation}

The characters of $SO(1,1)$ are
\begin{equation}
  \widehat{SO(1,1)} = \{\chi_{\epsilon,\nu} \mid \epsilon \in {\mathbb
    Z}/2{\mathbb Z}, \nu \in {\mathbb C}\}, \qquad
  \chi_{\epsilon,\nu}(r) = |r|^\nu \cdot \sgn(r)^\epsilon.
\end{equation}
We define
\begin{equation}
  \pi_{\epsilon,\nu}^{O(p,q)} =
  \Ind_{P^{O(p,q)}}^{O(p,q)}\left(\chi_{\epsilon,\nu}\otimes
  1_{O(p-1,q-1)}\otimes 1_{N^{O(p,q)}}\right).
\end{equation}
Here (in contrast to the definition of discrete series
$\pi^{O(p,q)}_\ell$) we use normalized induction, with a $\rho$
shift. As a consequence, the infinitesimal character of this
representation is
\begin{equation}
  \text{infl char}(\pi_{\epsilon,\nu}^{O(p,q)}) = (\nu, (p+q-4)/2,
  (p+q-6)/2,\cdots);
  \end{equation}
The continuous part of the Plancherel decomposition is
\begin{equation}\label{eq:Opqcont}
  L^2(H_{p,q})_{\text{cont}} = \sum_{\epsilon \in {\mathbb
      Z}/2{\mathbb Z}} \int_{\nu\in i{\mathbb R_{\ge 0}}}
    \pi_{\epsilon,\nu}^{O(p,q)}.
\end{equation}
Just as for the discrete part of the decomposition, {\em all} (not
just almost all) of the representations $\pi_{\epsilon,\nu}^{O(p,q)}$
are irreducible (always for $\nu \in i{\mathbb R}$).

There is an orbit-theoretic formulation of these parameters as well,
corresponding to elements $-i\nu\cdot(e_{1,p+1} + e_{p+1,1})/2$ of the split
Cartan subspace. We omit the details.
\end{subequations} 

\begin{subequations}\label{se:OKbranch}
We will need to understand the restriction of $\pi^{O(p,q)}_\ell$ to
the maximal compact subgroup
\begin{equation}
  K = O(p)\times O(q) \subset O(p,q).
\end{equation}
This computation requires knowing
\begin{equation}
  L^{O(p,q)} \cap K = SO(2)\times O(p-2)\times O(q), \qquad {\mathfrak
    u\cap s} = \chi_{1}\otimes 1 \otimes {\mathbb C}^{q};
\end{equation}
here ${\mathfrak g} = {\mathfrak k}\oplus {\mathfrak s}$ is the
complexified Cartan decomposition. Consequently
\begin{equation}
  S^m({\mathfrak u\cap s}) = \chi_m \otimes 1 \otimes S^m({\mathbb
    C}^{q}) = \sum_{0\le k \le m/2} \chi_m \otimes 1 \otimes \pi^{O(q)}_{m-2k}.
\end{equation}
Now an analysis of the Blattner formula for restricting
cohomologically induced representations to $K$ gives
\begin{equation}\label{eq:Obranch}
  \pi_\ell^{O(p,q)}|_{O(p)\times O(q)} =
    \sum_{m=0}^\infty \quad \sum_{0\le k \le m/2}\pi^{O(p)}_{m+\ell + q}
    \otimes \pi^{O(q)}_{m-2k}.
\end{equation}
If $p$ is much larger than $q$, then
some of the parameters for representations of $O(p)$ are negative.
Those representations should be understood to be zero.

A description of the restriction to $O(p)\times O(q)$ is in
\cite{RossH}*{Lemma 11}. In Rossmann's coordinates, what is written is
$$ \begin{aligned}
 \{ \pi_m^{O(p)} \otimes \pi_n^{O(q)} &\mid -(m+(p-2)/2) +
 (n+(q-2)/2) \ge \nu, \\ m+n &\equiv \nu -\rho - p \pmod{2}\}.
 \end{aligned}$$
Converting to our coordinates as explained after \eqref{eq:Opqdisc}
gives
\begin{equation}\begin{aligned}
\{ \pi_m^{O(p)} \otimes \pi_n^{O(q)} &\mid (m+(p-2)/2) -
(n+(q-2)/2) \ge \nu, \\
m-n &\equiv \nu -\rho + q \pmod{2}\},
\end{aligned}\end{equation}
or equivalently
\begin{equation}
\{ \pi_m^{O(p)} \otimes \pi_n^{O(q)} \mid m - n \ge \ell + q-1, \quad
m-n \equiv \ell + q \pmod{2}\}.
\end{equation}
The congruence condition makes the inequality into
$$m-n \ge \ell + q,$$
which matches the description in \eqref{eq:Obranch}

Finally, we record the easier formulas
\begin{equation}\label{eq:Ocontbranch}
  \pi_{\epsilon,\nu}^{O(p,q)}|_{O(p)\times O(q)} =
    \sum_{\substack{m,m' \ge 0\\ m-m' \equiv \epsilon \pmod{2}}} \pi^{O(p)}_m
    \otimes \pi^{O(q)}_{m'}.
\end{equation}
\end{subequations} 

\section{Hermitian hyperboloids}
\label{sec:Upq}
\setcounter{equation}{0}

\begin{subequations}\label{se:Upq}
In this section we see what the ideas from Sections \ref{sec:invt} and
\ref{sec:Opq} say about the discrete series of the non-symmetric
spherical spaces
\begin{equation}
  H_{2p,2q} = \{v\in {\mathbb C}^{p,q} \mid \langle
    v,v\rangle_{p,q} = 1\} = U(p,q)/U(p-1,q).
\end{equation}
Here $\langle,\rangle_{p,q}$ is the standard Hermitian form of
signature $(p,q)$ on ${\mathbb C}^{p+q}$. The inclusion of the right side
in the middle is just given by the action of the
unitary group on the basis vector $e_1$; surjectivity is Witt's
theorem for Hermitian forms. These discrete series were
completely described by Kobayashi in \cite{Kob}*{Theorem 6.1}.

To simplify many formulas, we write in this section
\begin{equation}\label{eq:Unpq}
  n = p+q.
\end{equation}

Our approach (like Kobayashi's) is to restrict the discrete series
representations $\pi_\ell^{O(2p,2q)}$ of \eqref{eq:Opqdisc} to
$U(p,q)$.

We should mention at this point that the homogeneous space
$U(n)/U(n-1)$ has another noncompact real form $GL(n,{\mathbb
  R})/GL(n-1,{\mathbb R})$, arising from the inclusion
\begin{equation}
  GL(n,{\mathbb R}) \hookrightarrow O(n,n)
\end{equation}
as a real Levi subgroup. For this real form (as Kobayashi observes)
the discrete series representations $\pi_\ell^{O(n,n)}$ decompose
continuously on restriction to $GL(n,{\mathbb R})$, and consequently
this homogeneous space has no discrete series. (More precisely, the
character $x-y$ of the center of $U(1)$ of $U(p,q)$ (an integer)
appearing in the analysis below must be replaced by a character of the
center ${\mathbb R}^\times$ of $GL(n,{\mathbb R})$ (a real number and
a sign).)

We begin by computing the restriction to $U(p)\times U(q)$.
What is good about this is that the representations of $O(2p)$ and
$O(2q)$  appearing in 
\eqref{eq:Obranch} are representations appearing in the action of $O$
on spheres. We already computed (in Theorem \ref{thm:OUcptbranch}) how
those branch to unitary groups. The conclusion is
\begin{small}\begin{equation}\label{e:OUbranch}
    \pi^{O(2p,2q)}_{_\ell}|_{U(p)\times U(q)} = 
    \sum_{\substack{0\le b,c \\[.2ex] b+c \ge \ell+2q}} \quad
    \sum_{\substack{0\le b',c' \\[.2ex] b'+c' \le b+c -\ell-2q
        \\[.2ex] b'+c'
        \equiv b+c-\ell \pmod{2}}}
\pi^{U(p)}_{b,c}\otimes \pi^{U(q)}_{b',c'}.
\end{equation}\end{small}
\end{subequations} 

This calculation, together with Corollary \ref{cor:Cdiff}, proves most of
\begin{proposition}\label{prop:OUbranch} Suppose $p$ and $q$ are
  nonnegative integers, each at least two; and suppose $\ell >
  -(n-1)$. Then the  restriction of the discrete series representation
  $\pi^{O(2p,2q)}_\ell$ to $U(p,q)$ is the direct sum of
  the one-parameter family of representations
  $$\pi^{U(p,q)}_{x,y},\quad 
  x,y \in {\mathbb Z}, \quad x+y = \ell.$$
  The infinitesimal character of $\pi^{U(p,q)}_{x,y}$ corresponds to
  the weight
  $$(x+(n-1)/2,(n-3)/2,\ldots,-(n-3)/2, -y-(n-1)/2).$$
  Restriction to the maximal compact subgroup is
  $$\pi^{U(p,q)}_{x,y}|_{U(p)\times U(q)} = \sum_{r,s\ge
    0}\sum_{k=0}^{\min(r,s)} \pi^{U(p)}_{x+q+r,y+q+s}\otimes
  \pi^{U(q)}_{s-k,r-k}.$$
  If one of the two subscripts in a $U(p)$ representation is negative,
  that term is to be interpreted as zero.

  Each of the representations $\pi^{U(p,q)}_{x,y}$ is irreducible.
\end{proposition}
The ``one parameter'' referred to in the proposition is $x-y$; the
pair $(x,y)$ can be thought of as a single parameter because of the
constraint $x+y=\ell$. What we have done is sorted the representations
of $U(p)\times U(q)$ appearing in \eqref{e:OUbranch} according to the
character of the center $U(1)$ of $U(p,q)$; this character is
$(b-c)+(b'-c')$, and we call it $x-y$ in the rearrangement in
Proposition \ref{prop:OUbranch}. The corresponding representation of
$U(p,q)$ (the part of $\pi^{O(2p,2q)}_\ell$ where $U(1)$ acts by
$x-y$) is what we call $\pi^{U(p,q)}_{x,y}$. In order to prove most of
the proposition, we just need to check that the same representations
of $U(p)\times U(q)$ appear in \eqref{e:OUbranch} and in Proposition
\ref{prop:OUbranch}, and this is easy. We will prove the
irreducibility assertion (using \cite{Kob}) after \eqref{e:Aq-smallx}
below.

\begin{subequations}\label{se:Upqrep}
  Having identified the restriction to $U(p)\times U(q)$, we record
  for completeness Kobayashi's identification of the actual representations of
  $U(p,q)$. These come in three families, according to the values of
  the integers $x$ and $y$. The families are cohomologically induced
  from three $\theta$-stable parabolic subalgebras:
\begin{equation}
  {\mathfrak q}^{U(p,q)}_+ = {\mathfrak l}^{U(p,q)}_+ + {\mathfrak
    u}^{U(p,q)}_+  \subset {\mathfrak u}(n,{\mathbb C});
\end{equation}
with Levi subgroup
\begin{equation}
  L^{U(p,q)}_+ = U(1)_p\times U(1)_q \times U(p-1,q-1);
\end{equation}
\begin{equation}
  {\mathfrak q}^{U(p,q)}_0 = {\mathfrak l}^{U(p,q)}_0 + {\mathfrak
    u}^{U(p,q)}_0  \subset {\mathfrak u}(n,{\mathbb C});
\end{equation}
with Levi subgroup
\begin{equation}
  L^{U(p,q)}_0 = U(1)_p \times U(p-2,q)\times U(1)_p;
\end{equation}
and
\begin{equation}
  {\mathfrak q}^{U(p,q)}_- = {\mathfrak l}^{U(p,q)}_- + {\mathfrak
    u}^{U(p,q)}_-  \subset {\mathfrak u}(n,{\mathbb C});
\end{equation}
with Levi subgroup
\begin{equation}
  L^{U(p,q)}_- = U(p-1,q-1)\times U(1)_q \times U(1)_p.
\end{equation}
(We write $U(1)_p$ for a coordinate $U(1) \subset U(p)$, and $U(1)_q
\subset U(q)$ similarly. More complete descriptions of these
parabolics are in \cite{Kob}.)
Suppose first that
\begin{equation}\label{eq:bigx}
  x> \ell + (n-1)/2, \qquad y < -(n-1)/2.
\end{equation}
(Since $x+y=\ell$, these two inequalities are equivalent.) Write
$\xi_x$ for the character of $U(1)$ corresponding to $x\in {\mathbb Z}$.
Consider the one-dimensional character
\begin{equation}
  \lambda^+_{x,y} = \xi_x\otimes \xi_{-(y + n-2)} \otimes {\det}^1
\end{equation}
of $L^{U(p,q)}_+$. What Kobayashi proves in \cite{Kob}*{Theorem 6.1}
is
\begin{equation}
  \pi^{U(p,q)}_{x,y} = A_{{\mathfrak q}^{U(p,q)}_+}(\lambda^+_{x,y})
  \qquad (x > \ell + (n-1)/2).
\end{equation}
Suppose next that
\begin{equation}\label{eq:middlex}
 \ell + (n-1)/2 \ge x \ge -(n-1)/2, \qquad  -(n-1)/2\le y \le
 \ell + (n-1)/2.
\end{equation}
(Since $x+y=\ell$, these two pairs of inequalities are equivalent.)
Consider the one-dimensional character
\begin{equation}
  \lambda^0_{x,y} = \xi_x\otimes 1\otimes \xi_{-y}
\end{equation}
of $L^{U(p,q)}_0$. Kobayashi's result in \cite{Kob}*{Theorem 6.1}
is now
\begin{equation}
  \pi^{U(p,q)}_{x,y} = A_{{\mathfrak q}^{U(p,q)}_0}(\lambda^0_{x,y})
  \qquad (-(n-1)/2 \le x \le \ell + (n-1)/2).
\end{equation}
The remaining case is
\begin{equation}\label{eq:smallx}
  x < -(n-1)/2, \qquad y > \ell + (n-1)/2.
\end{equation}
(Since $x+y=\ell$, these two inequalities are equivalent.) Write
\begin{equation}
  \lambda^-_{x,y} = {\det}^{-1}\otimes \xi_{x+n-2}\otimes \xi_{-y}
\end{equation}
of $L^{U(p,q)}_-$. In this case Kobayashi proves
\begin{equation}\label{e:Aq-smallx}
  \pi^{U(p,q)}_{x,y} = A_{{\mathfrak q}^{U(p,q)}_-}(\lambda^-_{x,y})
  \qquad (x < -(n-1)/2).
\end{equation}
\end{subequations}

\begin{subequations} \label{se:Upqorbit}
Here is the orbit method perspective. Just as for $U(n)$, we use a
trace form to identify ${\mathfrak g}_0^*$ with ${\mathfrak
  g}_0$. The linear functionals vanishing on ${\mathfrak h}_0^*$ are
\begin{equation}
  \lambda(t,u,v) = \begin{pmatrix}it & \hskip -2ex u& \hskip -2ex
    v\\[.5ex] -\overline u &
    \\ \overline v & \text{\Large \hskip 3ex $0_{(n-1)\times (n-1)}$ \hskip
      -2ex}\\ \end{pmatrix} \simeq
         {\mathbb R} + {\mathbb C}^{p-1,q}
\end{equation}
with $t\in {\mathbb R}$, $u \in {\mathbb C}^{p-1}$, $v\in {\mathbb C}^q$.

The orbits of $H=U(p-1,q)$ of largest dimension are given by the real
number $t$, and the value of the Hermitian form on the vector $(u,v)$:
positive for the orbits represented by nonzero elements
$r(e_{12}-e_{21})$ (nonzero eigenvalues $i(t\pm a)/2$, with $a=(t^2 +
4r^2)^{1/2}$); negative for nonzero elements $s(e_{1,p+1}+ e_{p+1,1})$
(nonzero eigenvalues $i(t\pm a)/2$, with $a=(t^2 - 4s^2)^{1/2}$); and
zero for the nilpotent element $(e_{12}-e_{21}+e_{1,p+1}+e_{p+1,1})$
(two nonzero eigenvalues $it/2$).

Define
\begin{equation}
  \ell_{\text{orbit}} = \ell + (n-1), \quad x_{\text{orbit}} = x +
  (n-1)/2, \quad y_{\text{orbit}} = y + (n-1)/2.
  \end{equation}
The coadjoint orbits for discrete series have representatives
\begin{equation}
    \lambda(x_{\text{orbit}},y_{\text{orbit}}) = \begin{cases} ix_{\text{orbit}}e_1
    - iy_{\text{orbit}}e_{p+1} & x_{\text{orbit}} > 0 >
    y_{\text{orbit}}\\
    ix_{\text{orbit}}e_1  + (e_{2,p}-e_{p,2}) & \\
    \quad+ (e_{2,p+1}+e_{p+1,2}) & x_{\text{orbit}} > 0 =  y_{\text{orbit}}\\
    ix_{\text{orbit}}e_1 - iy_{\text{orbit}}e_{p} & x_{\text{orbit}} >
    y_{\text{orbit}} > 0\\
    iy_{\text{orbit}}e_p + (e_{1,2}-e_{2,1})&\\
   \quad +e_{1,p+1}+e_{p+1,1}) &  x_{\text{orbit}} = 0 > y_{\text{orbit}}\\
    ix_{\text{orbit}}e_p - iy_{\text{orbit}}e_{p+q} & 0 > x_{\text{orbit}} >
    y_{\text{orbit}}.
    \end{cases}
\end{equation}
Then
\begin{equation}
  \pi_{x,y}^{U(p,q)} = \pi(\text{orbit}\ \lambda(x_{\text{orbit}},y_{\text{orbit}})).
\end{equation}
(We have not discussed attaching representations to partly nilpotent
coadjoint orbits like $\lambda(x_{\text{orbit}},0)$ (with
$x_{\text{orbit}} >0$); suffice it to say that the definitions given
above using ${\mathfrak q}_0$ are reasonable ones. It would be equally
reasonable to use instead ${\mathfrak q}_+$. We will see in
\eqref{eq:Aq-big-middle-coincide} that this leads to the same
representation.)

In the orbit method picture the condition \eqref{eq:bigx} simplifies to
\begin{equation}\label{eq:bigxorbit}
  x_{\text{orbit}} >  y_{\text{orbit}} > 0.
\end{equation}
Similarly, \eqref{eq:smallx} becomes
\begin{equation}\label{eq:smallxorbit}
  x_{\text{orbit}} < y_{\text{orbit}}<0.
\end{equation}
Finally, \eqref{eq:middlex} is
\begin{equation}\label{eq:middlexorbit}
  x_{\text{orbit}} \ge 0 \ge  y_{\text{orbit}};
\end{equation}
equality in either of these inequalities is the case of partially
nilpotent coadjoint orbits.
In all cases we need also the genericity condition
\begin{equation}
  \ell_{\text{orbit}} >0 \iff \ell > -(n-1),
\end{equation}
and the integrality conditions
\begin{equation}
  x_{\text{orbit}} \equiv (n-1)/2 \pmod{\mathbb Z},  y_{\text{orbit}}
  \equiv (n-1)/2 \pmod{\mathbb Z}.
\end{equation}
\end{subequations} 

\begin{subequations}
Here now is a sketch of a proof of the irreducibility assertion from
Proposition \ref{prop:OUbranch}. Each of the cohomologically induced
representations above is in the
weakly fair range.  The general results for the weakly fair range of
\cite{Vunit} together with \cite{VGLn}*{Section 16} apply to show that
they are irreducible or zero.  The key point is that the moment map
for the cotangent bundle to a relevant partial flag variety is
birational onto its image. This is automatic in type $A$, which is why
the arguments in \cite{VGLn} for $GL(n,\mathbb{R})$ also apply to
$U(p,q)$.

We close with a comment about how the three series of derived functor
modules fit together.  If we relax the strict inequalities on $x$ (and
$y$) in \eqref{eq:bigx}, then we are at one edge of the weak
inequalities in \eqref{eq:middlex}.  For these values of $x$ and $y$
(which occur only when $n$ is odd), namely
\[
(x,y) = \left ( \ell + (n-1)/2, -(n-1)/2 \right),
\]
or equivalently
\[
(x_{\text{orbit}},y_{\text{orbit}}) = \left(
\ell_{\text{orbit}},0\right),
\]
we claim
\begin{equation}\label{eq:Aq-big-middle-coincide}
A_{{\mathfrak q}^{U(p,q)}_+}(\lambda^+_{x,y}) \simeq A_{{\mathfrak
    q}^{U(p,q)}_0}(\lambda^0_{x,y}).
\end{equation}
To see this, one can begin by checking they have the same associated
variety: the $U(p,{\mathbb C}) \times U(q,{\mathbb C})$ saturations of
${\mathfrak u}^{U(p,q)}_+ \cap \mathfrak{s}$ and ${\mathfrak
  u}^{U(p,q)}_0 \cap \mathfrak{s}$ coincide.  (The dense orbit of
$U(p,{\mathbb C}) \times U(q,{\mathbb C})$ is one of the two
possibilities with one Jordan block of size $3$ and the others of size
$1$.)  A little further checking shows that they also have the same
annihilator: for $\ell \leq (n-2)/2$, given the associated variety
calculation, there is a unique possibility for the annihilator; a
slightly more refined analysis handles larger $\ell$.  Given that
their annihilators and associated varieties are the same, the main
result of \cite{BVUpq} implies \eqref{eq:Aq-big-middle-coincide}.
Similarly, for the other edge of the inequalities in
\eqref{eq:middlex}, namely
\[
(x,y) = \left (-(n-1)/2, \ell+(n-1)/2 \right),
\]
we have
\begin{equation}\label{eq:Aq-small-middle-coincide}
A_{{\mathfrak q}^{U(p,q)}_0}(\lambda^0_{x,y}) \simeq A_{{\mathfrak
    q}^{U(p,q)}_-}(\lambda^-_{x,y})
\end{equation}
by a similar argument.
\end{subequations}

\section{Quaternionic hyperboloids}
\label{sec:Sppq}
\setcounter{equation}{0}

\begin{subequations}\label{se:Sppq}
In this section we use the ideas from Section \ref{sec:invt} to investigate the
discrete series of the non-symmetric spherical spaces
\begin{equation}\begin{aligned}
  H_{4p,4q} &= \{v\in {\mathbb H}^{p,q} \mid \langle
    v,v\rangle_{p,q} = 1\}\\ &= [Sp(p,q)\times Sp(1)]/[Sp(p-1,q)\times
      Sp(1)_\Delta].
\end{aligned}\end{equation}
Here $\langle,\rangle_{p,q}$ is the standard Hermitian form of
signature $(p,q)$ on ${\mathbb H}^{p+q}$. We are using the action of
a real form of the enlarged group from \eqref{eq:Spbig}, namely
\begin{equation}
  Sp(p,q) \times Sp(1) = Sp(p,q)_{\text{linear}} \times
  Sp(1)_{\text{scalar}};
\end{equation}

The inclusion of the last
side of the equality (for $H_{4p,4q}$) in the middle is just given by
the action of this enlarged quaternionic unitary group on the basis
vector $e_1$; surjectivity is Witt's theorem for quaternionic
Hermitian forms. To avoid talking about degenerate cases, we will
assume
\begin{equation}
  p, q \ge 2.
  \end{equation}
Just as in Section \ref{sec:Upq}, we will simplify many formulas by
writing
\begin{equation}\label{eq:Spnpq}
  n = p+q.
\end{equation}

The homogeneous space
$Sp(n)/Sp(n-1)$ has another noncompact real form $[Sp(2n,{\mathbb
  R})\times Sp(2,{\mathbb R})]/[Sp(2(n-1),{\mathbb R})\times
  Sp(2,{\mathbb R})_\Delta]$, arising from an inclusion
\begin{equation}
  Sp(2n,{\mathbb R})\times Sp(2,{\mathbb R})\ \hookrightarrow O(2n,2n).
\end{equation}
This real form certainly has discrete series: we expect that the
discrete summands of the restriction of $\pi_\ell^{O(2n,2n)}$ are
indexed by discrete series representations of $Sp(2,{\mathbb R})$,
just as we find below (for $Sp(p,q)$) that they are indexed by
irreducible representations of the compact group $Sp(1)$. But we have
not carried out this analysis.

Our goal is to restrict the discrete series representations
$\pi^{O(4p,4q)}_\ell$ of \eqref{eq:Opqdisc} to $Sp(p,q)$, and so to
understand some representations in the discrete series of $(Sp(p,q)\times
Sp(1))/(Sp(p-1,q)\times Sp(1))$.
\end{subequations} 

\begin{subequations}\label{se:OSpbranch}
We have calculated in Theorem \ref{thm:OSpcptbranch} how the $O(4p)$
and $O(4q)$ representations appearing in \eqref{eq:Obranch} restrict
to $Sp$. The result is
\begin{small}\begin{equation}\label{e:OSpBigBranch}
    \pi_\ell^{O(4p,4q)}|_{[Sp(p)\times Sp(1)]\times [Sp(q)\times Sp(1)]} =
\sum_{\substack {m=0\\[.1ex] 0 \le k \le m/2}}^\infty \
\sum_{\substack{0\le e \le d\\[.1ex] 0\le e'\le d'\\[.2ex] d+e =
    m+\ell+4q\\[.1ex] d'+e' = m-2k}}
\pi^{Sp(p)\times Sp(1)}_{d,e}\otimes \pi^{Sp(q)\times Sp(1)}_{d',e'}.
\end{equation}\end{small}
The group to which we are restricting here is actually a little larger
than the maximal compact subgroup of $Sp(p,q)\times Sp(1)$, which is
\begin{equation}
  Sp(p) \times Sp(q) \times Sp(1)_\Delta;
\end{equation}
the subscript $\Delta$ indicates that this $Sp(1)$ factor
(corresponding to scalar multiplication on ${\mathbb H}^{p,q})$) is
diagonal in the $Sp(1)\times Sp(1)$ of \eqref{e:OSpBigBranch}
(corresponding to separate scalar multiplications on ${\mathbb H}^p$
and ${\mathbb H}^q$).
The branching $(G\times G)|_{G_\Delta}$ is tensor product
decomposition, which is very simple for $Sp(1)$. We find
\begin{small}\begin{equation}\label{e:OSpcptBranch}\begin{aligned}
    &\pi_\ell^{O(4p,4q)}|_{[Sp(p)\times Sp(q)\times Sp(1)]} =\\
&\sum_{\substack {m=0\\[.2ex] 0 \le k \le m/2}}^\infty \
\sum_{\substack{0\le e \le d\\[.1ex] 0\le e'\le d'\\[.2ex] d+e =
    m+\ell+4q\\[.1ex] d'+e' = m-2k}}
\sum_{j=0}^{\min(d-e,d'-e')}
\pi^{Sp(p)}_{d,e}\otimes \pi^{Sp(q)}_{d',e'}\otimes \pi^{Sp(1)}_{d+d'-e-e'-2j}.
\end{aligned}\end{equation}\end{small}

It will be useful to rewrite this formula. The indices $m$ and $k$
serve only to bound some of the other indices, so we can eliminate
them by rewriting the bounds. We find
\begin{small}\begin{equation}\label{e:OSpcptBranch2}\begin{aligned}
    &\pi_\ell^{O(4p,4q)}|_{[Sp(p)\times Sp(q)\times Sp(1)]} =\\[.5ex]
&\sum_{\substack{0\le e \le d\quad 0\le e'\le d'\\[.4ex] d'+e' \le d+e-\ell-4q
          \\[.4ex] d'+e' \equiv d+e - \ell \pmod{2}}}
     \sum_{\substack{|(d-e)-(d'-e')|\le f \\[.2ex] \quad \le
         (d-e)+(d'-e')\\[.4ex] \quad f\equiv
         (d-e)+(d'-e')\pmod{2}}}
\pi^{Sp(p)}_{d,e}\otimes \pi^{Sp(q)}_{d',e'}\otimes \pi^{Sp(1)}_f.
\end{aligned}\end{equation}\end{small}

For each of these representations of $K$, define integers $x$ and $y$
by solving the equations
\begin{equation}
  x+y = \ell, \qquad x-y = f.
\end{equation}
The congruence condition on $f$ guarantees that $x$ and $y$ are
indeed integers. Conversely, given any integers $x$ and $y$ satisfying
\begin{equation}
  x+y=\ell, \qquad x\ge y
\end{equation}
we can define
\begin{equation}
\begin{split}
  \pi^{Sp(p,q)\times Sp(1)}_{x,y} = \text{subrepresentation of
    $\pi^{O(4p,4q)}_\ell|_{Sp(p,q)\times Sp(1)}$}\\
    \; \text{where $Sp(1)$ acts
    with infl. char. $x-y+1$.}
\end{split}
\end{equation}
Equivalently, we are asking that $Sp(1)$ act by a multiple of
$\pi^{Sp(1)}_{x-y}$.
\end{subequations} 

This calculation, together with Corollary \ref{cor:Hdiff}, proves most of
\begin{proposition}\label{prop:OSpbranch} Suppose $p$ and $q$ are
  nonnegative integers, each at least two; and suppose $\ell > -2n
  +1$. Then the  restriction of the discrete series representation
  $\pi^{O(4p,4q)}_\ell$ to $Sp(p,q)\times Sp(1)$ is the direct sum of
  the one-parameter family of representations
  $$\pi^{Sp(p,q)\times Sp(1)}_{x,y},\quad 
  x\ge y \in {\mathbb Z}, \quad x+y =
  \ell.$$
  The infinitesimal character of $\pi^{Sp(p,q)\times Sp(1)}_{x,y}$ corresponds to
  the weight
  $$(x+n, y+n-1,n-2,\ldots,1)(x-y+1).$$
  Restriction to the maximal compact subgroup is
 \begin{small}\begin{equation*}\begin{aligned}
    & \pi^{Sp(p,q)\times Sp(1)}_{x,y}|_{Sp(p)\times Sp(q) \times Sp(1)}  =\\[.5ex]
&\sum_{\substack{0\le e \le d\quad 0\le e'\le d'\\[.4ex] d'+e' \le d+e-(x+y)-4q
          \\[.4ex] d'+e' \equiv d+e - (x+y) \pmod{2}}}
     \sum_{\substack{|(d-e)-(d'-e')|\le x-y \\[.2ex] x-y \le
         (d-e)+(d'-e')\\[.4ex] \quad
         (d-e)+(d'-e') \equiv x-y \pmod{2}}}
\pi^{Sp(p)}_{d,e}\otimes \pi^{Sp(q)}_{d',e'}\otimes \pi^{Sp(1)}_{x-y}.
\end{aligned}\end{equation*}\end{small}

 Each of the representations $\pi^{Sp(p,q)}_{x,y}$ is irreducible.
\end{proposition}

We will prove the irreducibility assertions (using \cite{Kob}) after
\eqref{eq:Aq-Spsmallx} below.

\begin{subequations}\label{se:Sppqrep}
Having identified the restriction to $Sp(p)\times Sp(q)\times Sp(1)$,
we want to record Kobayashi's identification of the actual
representations of $Sp(p,q)\times Sp(1)$. These come in two
families, according to
the values of the integers $x$ and $y$. The families are
cohomologically induced from two $\theta$-stable parabolic
subalgebras. The first is
\begin{equation}
  {\mathfrak q}_+^{Sp(p,q)\times Sp(1)} = {\mathfrak l}_+^{Sp(p,q)\times
    Sp(1)} + {\mathfrak u}_+^{Sp(p,q)\times Sp(1)} \subset
  {\mathfrak s}{\mathfrak p}(n,{\mathbb C}) \times  {\mathfrak
    s}{\mathfrak p}(1,{\mathbb C}),
\end{equation}
with Levi subgroup
\begin{equation}
  L_+^{Sp(p,q)\times Sp(1)} = [U(1)_p\times U(1)_q\times Sp(p-1,q-1)] \times U(1).
\end{equation}
(The first three factors are in $Sp(p,q)$. We write $U(1)_p$ for a
coordinate $U(1) \subset U(p)$, and $U(1)_q \subset U(q)$ similarly.)
The second parabolic is
\begin{equation}
  {\mathfrak q}_0^{Sp(p,q)\times Sp(1)} = {\mathfrak l}_0^{Sp(p,q)\times
    Sp(1)} + {\mathfrak u}_0^{Sp(p,q)\times Sp(1)} \subset
  {\mathfrak s}{\mathfrak p}(n,{\mathbb C}) \times  {\mathfrak
    s}{\mathfrak p}(1,{\mathbb C}),
\end{equation}
with Levi subgroup
\begin{equation}
  L_0^{Sp(p,q)\times Sp(1)} = [U(1)_p\times U(1)_p\times Sp(p-2,q)] \times U(1).
\end{equation}
More complete descriptions of these parabolics are in \cite{Kob}.)
Suppose first that
\begin{equation}\label{eq:Spbigx}
  x> \ell + (n-1), \qquad y < -(n-1).
\end{equation}
(Since $x+y=\ell$, these two inequalities are equivalent.) Write
$\xi_x$ for the character of $U(1)$ corresponding to $x\in {\mathbb Z}$.
Consider the one-dimensional character
\begin{equation}\label{eq:lam-Spbigx}
  \lambda^+_{x,y} = \left[\xi_x\otimes \xi_{-(y+2n-2)} \otimes
    1\right]\otimes \xi_{x-y}
\end{equation}
of $L^{Sp(p,q)\times Sp(1)}_+$. What Kobayashi proves in \cite{Kob}*{Theorem 6.1}
is
\begin{equation}\label{eq:Aq-Spbigx}
  \pi^{Sp(p,q)\times Sp(1)}_{x,y} = A_{{\mathfrak q}^{Sp(p,q)\times
      Sp(1)}_+}(\lambda^+_{x,y}) \qquad x > \ell + (n-1).
\end{equation}
Suppose next that
\begin{equation}\label{eq:Spsmallx}
 \ell + (n-1) \ge x > \ell/2, \qquad  -(n-1)\le y < \ell/2.
\end{equation}
(Since $x+y=\ell$, these two pairs of inequalities are equivalent.)
Consider the one-dimensional character
\begin{equation}\label{eq:lam-Spsmallx}
  \lambda^0_{x,y} = \left[\xi_x\otimes \xi_y\otimes 1\right]\otimes \xi_{x-y}
\end{equation}
of $L^{Sp(p,q)\times Sp(1)}_0$. Kobayashi's result in \cite{Kob}*{Theorem 6.1}
is now
\begin{equation}\label{eq:Aq-Spsmallx}
  \pi^{Sp(p,q)\times Sp(1)}_{x,y} = A_{{\mathfrak q}^{Sp(p,q)\times
      Sp(1)}_0}(\lambda^0_{x,y})
  \qquad \ell/2 < x \le \ell + (n-1)).
\end{equation}
\end{subequations}

\begin{subequations}\label{se:Sppqorbit}
Here is the orbit method perspective. Use a
trace form to identify ${\mathfrak g}_0^*$ with ${\mathfrak
  g}_0$. Linear functionals vanishing on ${\mathfrak h}_0^*$ are
quaternionic matrices
\begin{equation}
  \lambda(z,u,v) = \left[\begin{pmatrix}z & \hskip -2ex u& \hskip -2ex
    v\\[.5ex] -\overline u &
    \\ \overline v & \text{\Large \hskip 3ex $0_{(n-1)\times (n-1)}$ \hskip
      -2ex}\\ \end{pmatrix},-z\right] \simeq
         {\mathfrak s}{\mathfrak p}(1) + {\mathbb H}^{p-1,q}
\end{equation}
with $z\in {\mathfrak s}{\mathfrak p}(1)$ (the purely imaginary
quaternions), $u \in {\mathbb H}^{p-1}$, $v\in {\mathbb H}^q$.

The orbits of $H=Sp(p-1,q)\times Sp(1)_\Delta$ of largest dimension are given
by $|z|$, and the value of the Hermitian form on the vector $(u,v)$:
positive for the orbits represented by nonzero elements
$r(e_{12}-e_{21})$ (nonzero eigenvalues $i(|z|\pm a)/2$, with $a=(|z|^2 +
4r^2)^{1/2}$); negative for nonzero elements $s(e_{1,p+1}+ e_{p+1,1})$
(nonzero eigenvalues $i(|z|\pm a)/2$, with $a=(|z|^2 - 4s^2)^{1/2}$); and
zero for the nilpotent element $(e_{12}-e_{21}+e_{1,p+1}+e_{p+1,1})$
(nonzero eigenvalues $i|z|/2$).

Define
\begin{equation}
  \ell_{\text{orbit}} = \ell + (2n-1), \quad x_{\text{orbit}} = x +
  n, \quad y_{\text{orbit}} = y + n-1.
  \end{equation}
The coadjoint orbits for discrete series have representatives
\begin{equation}
    \lambda(x_{\text{orbit}},y_{\text{orbit}}) = \begin{cases} [ix_{\text{orbit}}e_1
    -iy_{\text{orbit}}e_{p+1},i(x_{\text{orbit}} - y_{\text{orbit}})]&\\
      \qquad x_{\text{orbit}} > 0 > y_{\text{orbit}}&\\[.5ex]
 [ix_{\text{orbit}}e_1 +
   (e_{2,p}-e_{p,2}+e_{2,p+1}+e_{p+1,2}),ix_{\text{orbit}}] &\\
 \qquad x_{\text{orbit}} > 0 =  y_{\text{orbit}}&\\[.5ex]
    [ix_{\text{orbit}}e_1 + iy_{\text{orbit}}e_2,i(x_{\text{orbit}} -
      y_{\text{orbit}})] & \\
    \qquad x_{\text{orbit}} > y_{\text{orbit}} > 0 &
    \end{cases}
\end{equation}
Then
\begin{equation}
  \pi_{x,y}^{Sp(p,q)} = \pi(\text{orbit}\ \lambda(x_{\text{orbit}},y_{\text{orbit}})).
\end{equation}
(The partly nilpotent
coadjoint orbits $\lambda(x_{\text{orbit}},0)$ (with
$x_{\text{orbit}} >0$) can be treated as for $U(p,q)$.)

In the orbit method picture the condition \eqref{eq:Spbigx} simplifies to
\begin{equation}\label{eq:Spbigxorbit}
  x_{\text{orbit}} >  0 > y_{\text{orbit}}.
\end{equation}
Similarly, \eqref{eq:Spsmallx} becomes
\begin{equation}\label{eq:Spsmallxorbit}
  x_{\text{orbit}} > y_{\text{orbit}} \ge 0;
\end{equation}
equality in the inequality is the case of partially nilpotent coadjoint orbits.
In all cases we need also the genericity condition
\begin{equation}
  \ell_{\text{orbit}} >0 \iff \ell > -(2n-1), \qquad
  x_{\text{orbit}}-y_{\text{orbit}}>0 \iff x-y+1 > 0
\end{equation}
and the integrality conditions
\begin{equation}
  x_{\text{orbit}} \equiv n \pmod{\mathbb Z},\qquad  y_{\text{orbit}}
  \equiv n-1 \pmod{\mathbb Z}.
\end{equation}
\end{subequations} 

\begin{subequations}\label{se:Sppqirr}
Here is a sketch of proof of the irreducibility assertion from
Proposition \ref{prop:OSpbranch}. Each of the cohomologically induced
representations above is in the weakly fair range, so the general
theory of \cite{Vunit} applies. One conclusion of this theory is that
the cohomologically induced representations are irreducible modules
for a certain twisted differential operator algebra ${\mathcal
  D}_{x,y}$; but in contrast to the $U(p,q)$ case, the natural map
$$U({\mathfrak s}{\mathfrak p}(p+q,{\mathbb C}) \times {\mathfrak
  s}{\mathfrak p}(1,{\mathbb C})) \rightarrow {\mathcal D}_{x,y}$$
{\em need not} be surjective: 
some of the cohomologically induced modules corresponding to discrete
series for $[Sp(2n,R)/Sp(2n-4,R)\times Sp(2,R)_\Delta$ are
  {\em reducible}.

Here is an irreducibility proof for the case
\eqref{eq:Aq-Spsmallx}. We begin by defining
\begin{equation}
  {\mathfrak q}_{0,\text{big}}^{Sp(p,q)\times Sp(1)} = {\mathfrak
    l}_{0,\text{big}}^{Sp(p,q)\times Sp(1)} + {\mathfrak
    u}_{0,\text{big}}^{Sp(p,q)\times Sp(1)} \supset
{\mathfrak q}_0^{Sp(p,q)\times Sp(1)}
\end{equation}
with Levi subgroup
\begin{equation}
  L_{0,big}^{Sp(p,q)\times Sp(1)} = [U(2)_p\times Sp(p-2,q)] \times U(1).
\end{equation}
Define
\begin{equation}
  \lambda^{0,\text{big}}_{x,y} = \left[\pi^{U(2)}_{x,y}\otimes
    1\right]\otimes \xi_{x-y}.
\end{equation}
Induction by stages proves that
\begin{equation}
  \pi^{Sp(p,q)\times Sp(1)}_{x,y} = A_{{\mathfrak q}^{Sp(p,q)\times
      Sp(1)}_{0,\text{big}}}(\lambda^{0,\text{big}}_{x,y})
  \qquad \ell/2 < x \le \ell + (n-1)).
\end{equation}
In this realization, the irreducibility argument from the $U(p,q)$
case goes through. The moment map from the cotangent bundle of the
(smaller) partial flag variety {\em is} birational onto its (normal)
image; so the map
\[U({\mathfrak s}{\mathfrak p}(p+q,{\mathbb C}) \times {\mathfrak
  s}{\mathfrak p}(1,{\mathbb C})) \rightarrow {\mathcal
  D}^{\text{small}}_{x,y}\]
is surjective, proving irreducibility. (The big parabolic subalgebra
defines a small partial flag variety, which is why we label the
twisted differential operator algebra ``small.'')

This argument does not apply to the case \eqref{eq:Aq-Spbigx},
since the corresponding larger Levi subgroup has a factor $U(1,1)$,
and the corresponding representation there is a discrete series. In
that case we have found only an unenlightening computational argument
for the irreducibility, which we omit.

Finally, the two series of derived functor modules fit together as
follows.  If we consider the edge of the inequalities in
\eqref{eq:Spbigx} and \eqref{eq:Spsmallx}, namely
\[
(x,y) = \left ( \ell + (n-1), -(n-1) \right ),
\]
then we have
\begin{equation}\label{eq:Aq-Sp-coincide}
 A_{{\mathfrak q}^{Sp(p,q)\times Sp(1)}_+}(\lambda^+_{x,y}) =
  A_{{\mathfrak q}^{Sp(p,q)\times Sp(1)}_0}(\lambda^0_{x,y}).
\end{equation}
For this equality, as for the irreducibility of $A_{{\mathfrak
    q}^{Sp(p,q)\times Sp(1)}_+}(\lambda^+_{x,y})$, we have found only
an unenlightening computational argument, which we omit.

\end{subequations}

\section{Octonionic hyperboloids}
\label{sec:spinp9-p}
\setcounter{equation}{0}

\begin{subequations}\label{se:spinp9-p}
We look for noncompact forms of the non-symmetric spherical space
$$S^{15} = \Spin(9)/\Spin(7)'$$
studied in Section \ref{sec:O}. The map from $\Spin(p,q)$ (with
$p+q=9$) to a form of
$O(16)$ will be given by the spin representation, which is therefore
required to be real. The spin representation is real if and only if
$p+q$ and $p-q$ are each congruent to $0$, $1$, or $7$ modulo $8$. The
candidates are
\begin{equation}
  G=\Spin(5,4) \quad\text{or}\quad G=\Spin(8,1),
\end{equation}
with maximal compact subgroups
\begin{equation}
  K = \Spin(5)\times_{\{\pm 1\}} \Spin(4) \quad\text{or}\quad \Spin(8);
\end{equation}
in the first case this means that the natural central subgroups $\{\pm 1\}$ in
$\Spin(5)$ and $\Spin(4)$ are identified with each other (and with the
natural central $\{\pm 1\}$ in $\Spin(5,4)$).  In each case the
sixteen-dimensional spin representation of $G$ is real and preserves a
quadratic form of signature $(8,8)$. One way to see this is to notice
that the restriction of the spin representation to $K$ is a sum
of two irreducible representations
  \begin{equation}
    \spin(5) \otimes \spin(4)_{\pm} \quad \text{or} \quad \spin(8)_\pm
  \end{equation}
Here $\spin(2m)_\pm$ denotes
the two half-spin representations, each of dimension $2^{m-1}$, of
$\Spin(2m)$. We are therefore looking at the hyperboloid
\begin{equation}\begin{aligned}
    H_{8,8} &= \{v\in {\mathbb R}^{8,8} \mid \langle v,v\rangle_{8,8} = 1\}\\
    &= \Spin(5,4)/\Spin(3,4)'\quad \text{or}\\
    &=\Spin(8,1)/\Spin(7)'.
\end{aligned}\end{equation}
\end{subequations}

\begin{subequations}\label{se:spin81}
The discrete series for the second case was described
by Kobayashi in connection with  branching from $SO(8,8)$ to $\Spin(8,1)$ in
\cite[Section 5.2]{toshi:howe}.  Here we carry out an 
approach
using the development above.
The harmonic analysis problem is
\begin{equation}
  L^2(H_{8,8}) \simeq L^2(\Spin(8,1))^{\Spin(7)'};
\end{equation}
the $\Spin(7)'$ action is on the right. This problem is resolved by
Harish-Chandra's Plancherel formula for $\Spin(8,1)$: the discrete
series are exactly those of Harish-Chandra's discrete series that
contain a $\Spin(7)'$-fixed vector, and the multiplicity is the
dimension of that fixed space. Because of Helgason's branching law
from $\Spin(7)'$ to $\Spin(8)$ \eqref{eq:87'}, the number in question is
the sum of the multiplicities of the $\Spin(8)$ representations of
highest weights
\begin{equation}
  \mu_y = (y/2,y/2,y/2,y/2) \qquad (y \in {\mathbb N}).
\end{equation}

Corollary \ref{cor:octdiff} constrains the possible infinitesimal
characters, and therefore the Harish-Chandra parameters, of
representations appearing on this hyperboloid. Here are the discrete
series having these infinitesimal characters. Suppose $x$ is an integer
satisfying $2x+y+7 >0$. Define
\begin{equation}
  \pi^{\Spin(8,1)}_{x,y,\pm} = \begin{cases} \substack{\text{discrete
        series with
        parameter}\\((2x+y+7)/2,(y+5)/2,(y+3)/2,\pm(y+1)/2)}& x \ge  0\\[.5ex]
    \qquad\qquad 0 & 0> x > -4\\[.5ex]
 \substack{\text{discrete series with
     parameter}\\ ((y+5)/2,(y+3)/2,(y+1)/2,\pm(2x+y+7)/2)} &
 -4 \ge x > -(y+7)/2.
\end{cases} \end{equation}
We can now use Blattner's formula to determine which of these discrete
series contain $\Spin(8)$ representations of highest weight
$\mu_y$. The representations with a subscript $-$ are immediately
ruled out (since the last coordinate of the highest weight of any
$K$-type of such a discrete series must be negative). Similarly, in
the first case with $+$ the lowest $K$-type has highest weight
$(2x+1,1,1,1)+\mu_y$, and all other highest weights of $K$-types arise
by adding positive integers to these coordinates; so $\mu_y$ cannot
arise.

In the third case with $+$ the lowest $K$-type has highest weight
$(0,0,0,x+4)+\mu_y$; we get to $\mu_y$ by adding the nonnegative
multiple $-x-4$ of the noncompact positive root $e_4$. A more careful
examination of Blattner's formula shows that in fact $\mu_y$ has
multiplicity one. This proves

\begin{equation}\label{eq:spin81ds}
  L^2(H_{8,8})_{\text{disc}} = \sum_{ 
  y\ge 1,\  -4 \ge x > -(y+7)/2} 
  \pi^{\Spin(8,1)}_{x,y,+}.
\end{equation}
Furthermore (by Corollary \ref{cor:octdiff})
\begin{equation}
  \pi^{O(8,8)}_\ell|_{\Spin(9,1)} = \sum_{\substack{ 
 y\ge 1,\  -4 \ge x \ge -(y+7)/2\\[.4ex] 2x+y=\ell}}
  \pi^{\Spin(8,1)}_{x,y,+}.
\end{equation}
These discrete series are cohomologically induced from
one-dimensional characters of the spin double cover of the compact
Levi subgroup
\begin{equation}
  SO(2)\times U(3) \subset SO(2) \times SO(6) \subset SO(8) \subset
  SO(8,1).
\end{equation}

\end{subequations} 

\begin{subequations}\label{se:spin81orbit}
Here is the orbit method perspective. We have
\[\label{eq:spin81orbits}
  ({\mathfrak g}_0/{\mathfrak h}_0)^* \simeq \Spin^8 + {\mathbb R}^7
\]
as a representation of $H=\Spin(7)'$; the first summand is the
$8$-dimensional spin representation. What distinguishes this from the
compact case analyzed in \eqref{se:Oorbit} is that the restriction of
the  natural $G$-invariant form has opposite signs on the two
summands; we take it to be negative on the first and positive on the
second. Because of \eqref{eq:GmodH},
the orbits we want are represented by $H$ orbits of maximal dimension
on this space. A generic orbit on ${\mathbb R}^7$ is given by the
value of the quadratic form
length $a_7 > 0$, and the corresponding isotropy group is
$\Spin(6)'\simeq SU(4)$. As a representation of $SU(4)$,
\[
\Spin^8 \simeq {\mathbb C}^4
\]
regarded as a real vector space. Here again the nonzero orbits are
indexed by the value of the Hermitian form $b_{\text{spin}} < 0$. The
conclusion is 
that the regular $H$ orbits on $({\mathfrak g}_0/{\mathfrak h}_0)^*$
are
\[
\lambda(a_7,b_{\text{spin}}) \qquad (a_7> 0, \quad b_{\text{spin}} < 0).
\]
It turns out that the eigenvalues of such a matrix are $\pm i(a_7/4)^{1/2}$
(repeated three times), $\pm i(a_7/4 +b_{\text{spin}})^{1/2}$, and
one more eigenvalue zero. Accordingly the element is elliptic if and
only if $a_7/4 + b_{\text{spin}} \ge 0$. In this case we write
\[
x_{\text{orbit}} = (a_7/4 + b_{\text{spin}})^{1/2} -a_7^{1/2}/2, \quad
y_{\text{orbit}} = a^{1/2}_7 \qquad (a_7/4 + b_{\text{spin}} \ge 0)
\]
The elliptic elements we want are
\[
\lambda(x_{\text{orbit}},y_{\text{orbit}}) = (y_{\text{orbit}}/2,
y_{\text{orbit}}/2, y_{\text{orbit}}/2, y_{\text{orbit}}/2 +x), \qquad
(y_{\text{orbit}}/2 > -x_{\text{orbit}} > 0);
\]
we have represented the element (in fairly standard coordinates) by
something in the dual of a compact Cartan subalgebra $[{\mathfrak
  s}{\mathfrak o}(2)]^4$ to which it is conjugate.

If now we define
\[
y = y_{\text{orbit}} -3, \quad x = x_{\text{orbit}} -2,
\]
then 
\begin{equation}
  \pi_{x,y,+}^{\Spin(8,1)} =
  \pi(\text{orbit}\ \lambda(x_{\text{orbit}},y_{\text{orbit}})) \qquad
  (0 > x_{\text{orbit}} > -y_{\text{orbit}}).
\end{equation}
When $y_{\text{orbit}} = 1$ or $2$ or $3$, or $x_{\text{orbit}}=-1$, these
representations are zero; that is the source of the conditions
$$y_{\text{orbit}} \ge 4, \quad -2 \ge x_{\text{orbit}} -
y_{\text{orbit}}/2$$
in \eqref{eq:spin81ds}.
\end{subequations} 

\begin{subequations}\label{se:spin54}
  In the first case of \eqref{se:spinp9-p}, we are looking at
\begin{equation}
H_{8,8} \simeq \Spin(5,4))/{\Spin(4,3)'};
\end{equation}
this is the ${\mathbb R}$-split version of Section \ref{sec:O}, and so
arises from
\begin{equation}
  \Spin(4,3)' \ {\buildrel{\text{spin}}\over
    \longrightarrow}\  \Spin(4,4) \subset \Spin(5,4).
\end{equation}
We have not determined the discrete series for this homogeneous space;
of course we expect two-parameter families of representations
cohomologically induced from one-dimensional characters of spin double
covers of real forms of $SO(2)\times U(3)$.

\end{subequations} 

\section{The split $G_2$ calculation}
\label{sec:G2s}
\setcounter{equation}{0}

\begin{subequations}\label{se:G2hyp}
Write $G_{2,s}$ for the $14$-dimensional split Lie group
of type $G_2$. There is a $7$-dimensional real representation
$(\tau_{{\mathbb R},s},W_{{\mathbb R},s})$ of
$G_{2,s}$, whose weights are zero and the six short
roots. This preserves an inner product of signature $(4,3)$, and so
defines an inclusion
\begin{equation}\label{eq:G2SO7nc}
  G_{2,s} \hookrightarrow SO(4,3).
\end{equation}
The corresponding actions of $G_{2,s}$ on the hyperboloids
\begin{equation}
  H_{4,3} = O(4,3)/O(3,3), \qquad H_{3,4} = O(3,4)/O(2,4)
\end{equation}
are transitive. The isotropy groups are real forms of $SU(3)$:
\begin{equation}
  H_{4,3} \simeq G_{2,s}/SL(3,{\mathbb R}), \qquad H_{3,4} \simeq
  G_{2,s}/SU(2,1).
\end{equation}
\end{subequations} 
The discrete series for these cases are given by Kobayashi (up to
two questions of reducibility) in
\cite[Thm 6.4]{Kob} ; see also \cite[Theorem 3.5]{toshi:zuckerman}.
We now give a self-contained treatment of the
classification, and resolve the reducibility.

\begin{subequations}\label{se:G2sreps}
The (real forms of) $O(7)$ representations appearing on these
hyperboloids are all related to the flag variety
\begin{equation}\begin{aligned}
    O(7,{\mathbb C})/P &= \text{isotropic lines in\ }{\mathbb C}^7, \\
     P = MN,\qquad M&=GL(1,{\mathbb C}) \times O(5,{\mathbb C}).
\end{aligned}\end{equation}
What makes everything simple is that $G_2({\mathbb C})$ is transitive
on this flag variety:
\begin{equation}\begin{aligned}
\text{isotropic lines in\ }{\mathbb C}^7 &= G_2({\mathbb C})/Q, \\
Q=LU, \qquad L &= GL(2,{\mathbb C}).
\end{aligned}\end{equation}
Precisely, the discrete series for $H_{4,3}$ are cohomologically
induced from the $\theta$-stable parabolic
\begin{equation}
  {\mathfrak p}_1 = {\mathfrak m}_1 + {\mathfrak n}_1, \qquad
  M_1=SO(2)\times O(2,3).
\end{equation}
The discrete series representations are
\begin{equation}
  \pi^{O(4,3)}_{1,\ell} = A_{{\mathfrak p}_1}(\lambda_1(\ell)), \qquad \ell+5/2> 0.
\end{equation}
(cf. \eqref{se:Upq}). The inducing representation is the $SO(2)$
character indexed by $\ell$, and trivial on $O(2,3)$. Similarly, the
discrete series for $H_{3,4}$ are cohomologically
induced from the $\theta$-stable parabolic
\begin{equation}
  {\mathfrak p}_2 = {\mathfrak m}_2 + {\mathfrak n}_2, \qquad
  M_2=SO(2)\times O(1,4).
\end{equation}
The discrete series are
\begin{equation}
 \pi^{O(3,4)}_{2,\ell} = A_{{\mathfrak p}_2}(\lambda_2(\ell)),\qquad \ell+5/2 > 0.
\end{equation}

The intersections of these parabolics with $G_2$ are
\begin{equation}
  {\mathfrak q}_1 = {\mathfrak l}_1 + {\mathfrak u}_1, \qquad
  L_1 = \text{long root $U(1,1)$}.
\end{equation}
and
\begin{equation}
  {\mathfrak q}_2 = {\mathfrak l}_2 + {\mathfrak u}_2, \qquad
  L_2= \text{long root $U(2)$}.
\end{equation}
(The Levi subgroups are just {\em locally} of this form.) 
Because the $G_2$ actions on the $O(4,3)$ partial flag varieties are
transitive, we get discrete series representations for $H_{4,3}$
\begin{equation}
  \pi^{G_{2,s}}_{1,\ell} = A_{{\mathfrak q}_1}(\lambda_1(\ell)),\qquad \ell+5/2 > 0.
\end{equation}
The character is $\ell$ times the action of $L_1$ on the highest short
root defining ${\mathfrak q}_1$.

Similarly, for the action on $H_{3,4}$
\begin{equation}
  \pi^{G_{2,s}}_{2,\ell} = A_{{\mathfrak q}_2}(\lambda_2(\ell)),\qquad \ell +5/2 > 0.
\end{equation}
The {\tt atlas} software \cite{atlas} tells us that all of these
discrete series representations of $G_2$ are irreducible, with the
single exception of $\pi^{G_{2,s}}_{1,-2} = A_{{\mathfrak
    q}_1}(\lambda_1(-2)).$ That representation is a sum of two
irreducible constituents. One constituent is the unique non-generic
limit of discrete series of infinitesimal character a short root. In
\cite{G2}*{Theorem 18.5}, (describing some of Arthur's unipotent
representations) this is the representation described in (b). The
other constituent is described in part (c) of that same theorem. The
irreducible representation $\pi^{G_{2,s}}_{2,-2} = A_{{\mathfrak
    q}_2}(\lambda_2(-2))$ appears in part (a) of the
theorem. All of these identifications (including the reducibility of
$\pi^{G_{2,s}}_{1,-2}$) follow from knowledge of the
$K$-types of these representations (given in \eqref{se:G2sK} below)
and the last assertion of \cite{G2}*{Theorem 18.5}.

Summarizing, in the notation of \cite{G2},
\begin{equation}\label{eq:G2unip}
  \pi^{G_{2,s}}_{1,-2} \simeq J_{-}(H_2; (2, 0)) \oplus J(H_2; (1,1)),
  \qquad   \pi^{G_{2,s}}_{2,-2} \simeq J(H_1; (1,1)).
  \end{equation}
That is, the first discrete series for these non-symmetric spherical
spaces include three of the five unipotent representations for the
split $G_2$ attached to the principal nilpotent in $SL(3) \subset
G_2$.
\end{subequations} 

\begin{subequations}\label{G2sorbit}
  Here is the orbit method perspective. For the case of $H_{4,3}$, the
  representation of $H=SL(3,{\mathbb R})$ on $[{\mathfrak
    g}_0/{\mathfrak h}_0]^*$ is ${\mathbb R}^3 + ({\mathbb
    R}^3)^*$. The generic orbits of $H$ are indexed by non-zero real
  numbers $A$, the value of a linear functional on a vector. We can
  arrange the normalizations so that the elliptic elements are exactly
  those with $A>0$; if we define
  \[ \ell_{\text{orbit}} = A^{1/2}, \qquad \ell = \ell_{\text{orbit}} - 5/2, \]
and write $\lambda_1(\ell_{\text{orbit}})$ for a representative of this
orbit, then
\[
    \pi_{1,\ell}^{G_{2,s}}  =
    \pi({\text{orbit}},\lambda_1(\ell_{\text{orbit}})) \qquad
      (\ell_{\text{orbit}} > 0).
\]

For the case of $H_{3,4}$, the representation of $H=SU(2,1)$ on $[{\mathfrak
    g}_0/{\mathfrak h}_0]^*$ is ${\mathbb C}^{2,1}$; generic orbits
are parametrized by the nonzero values $B$ of the Hermitian form of signature
$(2,1)$. The elliptic orbits are those with $B>0$; if we define
  \[ \ell_{\text{orbit}} = B^{1/2}, \qquad \ell = \ell_{\text{orbit}} - 5/2 \]
then
\[
    \pi_{2,\ell}^{G_{2,s}}  =
    \pi({\text{orbit}},\lambda_2(\ell_{\text{orbit}})) \qquad
      (\ell_{\text{orbit}} > 0).
\]
\end{subequations}

\begin{subequations} \label{se:G2sK}
We conclude this section by calculating the restrictions to
\begin{equation}
  K = SU(2)_{\text{long}} \times_{\{\pm 1\}} SU(2)_{\text{short}}
  \subset G_{2,s}.
\end{equation}
We define
\begin{equation}\begin{aligned}
  \gamma_d^{\text{long}} &= \text{$(d+1)$-diml irr of
    $SU(2)_{\text{long}}$}\\
    \gamma_d^{\text{short}} &= \text{$(d+1)$-diml irr of
      $SU(2)_{\text{short}}$}
\end{aligned}\end{equation}
The maximal compact of $O(4,3)$ is $O(4)\times O(3)$. The embedding of
$G_{2,s}$ sends $SU(2)_{\text{long}}$ to one of the factors in
$$O(4) \supset SO(4) \simeq SU(2)\times_{\{\pm 1\}} SU(2),$$
and sends $SU(2)_{\text{short}}$ diagonally into the product of the
other $SU(2)$ factor and $SO(3)\subset O(3)$ (by the two-fold cover
$SU(2) \rightarrow SO(3)$). According to \eqref{eq:Obranch},
\begin{equation}\begin{aligned}
\pi^{O(4,3)}_{1,\ell}|_{O(4)\times O(3)} = \sum_{\substack{d-\ell-3
    \ge e \ge 0 \\[.2ex] e\equiv d-\ell-3 \pmod{2}}} \pi^{O(4)}_d \otimes
\pi^{O(3)}_e\\
\pi^{O(3,4)}_{2,\ell}|_{O(3)\times O(4)} = \sum_{\substack{d' -\ell -4
    \ge e' \ge 0\\[.2ex] e'\equiv d'-\ell-4\pmod{2}}} \pi^{O(3)}_{d'} \otimes
\pi^{O(3)}_{e'}.
\end{aligned}\end{equation}
By an easy calculation, we deduce
\begin{equation}\begin{aligned}
\pi^{G_{2,s}}_{1,\ell}|_K = \sum_{\substack{d-\ell-3 \ge  e\ge 0\\[.2ex]
 e\equiv d-\ell-3 \pmod{2}}} \gamma_d^{\text{long}} \otimes
   \left[\gamma_d^{\text{short}} \otimes
     \gamma_{2e}^{\text{short}}\right].\\
\pi^{G_{2,s}}_{2,\ell}|_K = \sum_{\substack{d'-\ell-4 \ge e'\ge
    0\\[.2ex] e'\equiv d'-\ell-4\pmod{2}}}
\gamma^{\text{long}}_{e'} \otimes \left[ \gamma_{e'}^{\text{short}} \otimes
     \gamma_{2d'}^{\text{short}}\right]
\end{aligned}\end{equation}
The internal tensor products in the short $SU(2)$ factors are of
course easy to compute:
\begin{equation}
\pi^{G_{2,s}}_{1,\ell}|_K = \sum_{\substack{d-\ell-3 \ge e\ge
    0\\[.2ex] e\equiv d-\ell-3\pmod{2}}} \sum_{k=0}^{\min(d,2e)}
\gamma^{\text{long}}_{e'} \otimes \gamma_{d+2e-2k}^{\text{short}},
\end{equation}
\begin{equation}
\pi^{G_{2,s}}_{2,\ell}|_K = \sum_{\substack{d'-\ell-4 \ge e'\ge
    0\\[.2ex] e'\equiv d'-\ell-4\pmod{2}}} \sum_{k'=0}^{e'}
\gamma^{\text{long}}_{e'} \otimes \gamma_{2d'+e'-2k'}^{\text{short}}
\end{equation}
\end{subequations} 

\section{The noncompact big $G_2$ calculation}
\label{sec:bigncG2}
\setcounter{equation}{0}

\begin{subequations}\label{se:bigncG2hyp}
In this section we look at noncompact forms of $S^7 \simeq
\Spin(7)'/G_{2,c}$ from Section \ref{sec:bigG2}. The noncompact forms
of $\Spin(7)$ are $\Spin(p,q)$ with $p+q=7$, having maximal compact
subgroups $\Spin(p)\times_{\{\pm 1\}} \Spin(q)$. None of these compact
subgroups can contain $G_{2,c}$ (unless $pq=0$), so the isotropy subgroup we are
looking for is the split form $G_{2,s}$. The seven-dimensional
representation of $G_{2,s}$ is real, and its invariant bilinear form
is of signature $(3,4)$; so we are looking at
\begin{equation}
  G_{2,s} \hookrightarrow \Spin(3,4),
\end{equation}
the double cover of the inclusion \eqref{eq:G2SO7nc}. This homogeneous
space is discussed briefly in \cite{Kob}*{Corollary 5.6(e)}, which is proven
in part (ii) of the proof on page 197. We will argue along similar
lines, but get more complete conclusions (parallel to
Kobayashi's results described in Sections
\ref{sec:Upq}--\ref{sec:Sppq}).

The
eight-dimensional spin representation of $\Spin(3,4)$ is real and of
signature $(4,4)$, so we get
\begin{equation}
  \Spin(3,4)' \hookrightarrow \Spin(4,4), \qquad \Spin(3,4)'\cap
  \Spin(3,4) = G_{2,s}.
\end{equation}
The $\Spin(3,4)'$ action on
\begin{equation}
  H_{4,4} = \Spin(4,4)/\Spin(3,4)
\end{equation}
is transitive, so
\begin{equation}
  H_{4,4}\simeq \Spin(3,4)'/G_{2,s}.
\end{equation}
In a similar fashion, we find an identification of six-dimensional
complex manifolds
\begin{equation}\label{eq:qmatchbigG2}
  \Spin(4,4)/[\Spin(2)\times_{\{\pm 1\}}] \Spin(2,4)] \simeq
    \Spin(3,4)'/\widetilde{U(1,2)}.
\end{equation}
The manifold on the left corresponds to the $\theta$-stable parabolic
${\mathfrak q}^{O(4,4)}$ described in \eqref{eq:qOpq}; the discrete
series $\pi^{O(4,4)}_\ell$ for $H_{4,4}$ are obtained from it by
cohomological induction.

The manifold on the right corresponds to the $\theta$-stable parabolic

\begin{equation}\label{eq:qSpin34}
  {\mathfrak q}^{\Spin(3,4)'} = {\mathfrak l}^{\Spin(3,4)'} +
  {\mathfrak u}^{\Spin(3,4)'} \subset
  {\mathfrak o}(7,{\mathbb C});
\end{equation}
the corresponding Levi subgroup is
\begin{equation}
  L^{\Spin(3,4)'}  = \widetilde{U(1,2)}
\end{equation}
The covering here is the ``square root of determinant'' cover; the
one-dimensional characters are half integer powers of the
determinant. We are interested in
\begin{equation}\begin{aligned}
  \lambda_\ell &= {\det}^{\ell/2} \in [L^{\Spin(3,4)'}]\,\widehat{\ }
  \qquad (\ell +3 > 0).\\
  \pi^{\Spin(3,4)'}_\ell &= A_{{\mathfrak q}^{\Spin(3,4)}}(\lambda_\ell) \qquad
  (\ell > -3).
\end{aligned}\end{equation}
The infinitesimal character of this representation is
\begin{equation}
  \text{infl char}(\pi_\ell^{\Spin(4,3)'}) = ((\ell+5)/2,(\ell+3)/2,(\ell+1)/2).
\end{equation}
As a consequence of \eqref{eq:qmatchbigG2},
\begin{equation}
  \pi^{O(4,4)}_\ell|_{\Spin(3,4)'} \simeq \pi^{\Spin(3,4)'}_\ell.
\end{equation}

The discrete part of the Plancherel decomposition is therefore
\begin{equation}\label{eq:bigncG2disc}
  L^2(H_{4,4})_{\text{disc}} = \sum_{\ell > -3}
    \pi_\ell^{\Spin(3,4)'}.
\end{equation}
The ``weakly fair'' range for $\pi^{\Spin(3,4)'}_\ell$ is $\ell \ge
-3$, so all the representations $\pi^{\Spin(3,4)'}_\ell$ are
contained in the weakly fair range. In particular, \cite{Vunit}
establishes {\em a priori} the unitarity of what turn out to be the
discrete series representations.
But the results in \cite{Vunit} prove only
\begin{equation}\label{eq:bigG2unit}
  \text{$\pi^{\Spin(3,4)'}_\ell$ is irreducible for $\ell \ge 0$.}
\end{equation}
The {\tt atlas} software \cite{atlas} proves the irreducibility of the
first two discrete series (those not covered by \eqref{eq:bigG2unit}).

\end{subequations} 

\begin{subequations}\label{se:bigG2sorbit}
  Here is the orbit method perspective. The
  representation of $H=G_{2,s}$ on $[{\mathfrak
    g}_0/{\mathfrak h}_0]^*$ is ${\mathbb R}^{3,4}$, the real
  representation whose highest weight is a short 
  root. We have already said that this representation carries an
  invariant quadratic form of signature $(3,4)$. The generic orbits of
  $H$ are indexed by non-zero real
  numbers $A$, the values of the quadratic form. We can
  arrange the normalizations so that the elliptic elements are exactly
  those with $A>0$; if we define
  \[ \ell_{\text{orbit}} = A^{1/2}, \qquad \ell = \ell_{\text{orbit}} - 3, \]
and write $\lambda(\ell_{\text{orbit}})$ for a representative of this
orbit, then
\[
    \pi_\ell^{\Spin(3,4)}  =
    \pi({\text{orbit}},\lambda(\ell_{\text{orbit}})) \qquad
      (\ell_{\text{orbit}} > 0).
\]
\end{subequations} 

\end{document}